\documentclass[final]{siamltex}

\usepackage{graphicx}
\usepackage[colorlinks=true,citecolor=darkblue,urlcolor=darkblue,linkcolor=darkred]{hyperref}
\usepackage{amsmath,amsfonts,amssymb}
\usepackage{color}
\usepackage{verbatim}
\usepackage{xspace}
\usepackage{url}
\definecolor{lightgray}{gray}{0.5}
\definecolor{darkred}{rgb}{.6,0,0}
\definecolor{darkblue}{rgb}{0,0,0.5}

\usepackage{listings}
\graphicspath{{./}{figs/}}

\newcommand{\qed}{\hfill \ensuremath{\Box}}

\definecolor{dkgreen}{rgb}{0,0.4,0}
\definecolor{gray}{rgb}{0.5,0.5,0.5}
\definecolor{dkred}{rgb}{0.4,0.0,0.0}
\definecolor{dkblue}{rgb}{0,0,0.65}

\newtheorem{remark}{Remark}[section]

\newtheorem{example}{Example}[section]

\newcommand{\dx}{\Delta x}
\newcommand{\dt}{\Delta t}

\newcommand{\scinot}[2]{\ensuremath{#1\!\times\!10^{#2}}}

\newcommand{\s}[1]{Y^{(#1)}}
\newcommand{\m}[1]{#1}

\newcommand{\iph}{{i + \frac{1}{2}}}
\newcommand{\imh}{{i - \frac{1}{2}}}
\newcommand{\ipmh}{{i \pm \frac{1}{2}}}
\newcommand{\hf}{\frac{1}{2}}

\newcommand{\Dt}{\Delta t}
\newcommand{\Dx}{\Delta x}
\newcommand{\Iint}{{\mathcal I}}

\newcommand{\sspcoeff}{{\mathcal C}}

\newcommand{\Oop}{{\mathcal O}}

\newcommand{\tfinal}{t_\textup{f}}

\newcommand{\R}{\mathbb{R}}

\title{Spatially partitioned embedded Runge--Kutta methods}

\author{%
  David I.~Ketcheson\thanks{4700 King Abdullah University of Science \&
    Technology, Thuwal 23955, Saudi Arabia
    (\url{david.ketcheson@kaust.edu.sa}). The work of this author was
    supported by funding from King Abdullah University of Science and
    Technology (KAUST).}
  \and
  Colin B.~Macdonald\thanks{Mathematical Institute, University of
    Oxford, OX1\,3LB, UK (\url{macdonald@maths.ox.ac.uk}).  The work
    of this author was supported by NSERC Canada and
    Award No KUK-C1-013-04 made by King Abdullah
    University of Science and Technology (KAUST)}
  \and
  Steven J.~Ruuth\thanks{Department of Mathematics, Simon Fraser
    University, Burnaby, British Columbia, V5A\,1S6 Canada
    (\url{sruuth@sfu.ca}).  The work of this author was partially
    supported by a grant from NSERC Canada and by award KUK-C1-013-04 made by King Abdullah University of Science and Technology (KAUST).}%
}

\begin{document}

\maketitle

\begin{abstract}
We study spatially partitioned embedded Runge--Kutta (SPERK) schemes for partial differential equations (PDEs), 
in which each of the component schemes is applied over a different part of the spatial domain.
Such methods may be convenient for problems in which the smoothness of the solution or the 
magnitudes of the PDE coefficients vary strongly in space.  We focus on embedded
partitioned methods as they offer greater efficiency and avoid the order reduction
that may occur in non-embedded schemes.  We demonstrate that the lack of conservation
in partitioned schemes can lead to non-physical effects and propose conservative additive
schemes based on partitioning the fluxes rather than the ordinary differential equations.
A variety of SPERK schemes are presented, including an embedded pair suitable for the
time evolution of fifth-order weighted non-oscillatory (WENO) spatial discretizations.
Numerical experiments are provided to support the theory.    
\end{abstract}



\pagestyle{myheadings}
\thispagestyle{plain}
\markboth{Ketcheson, Macdonald and Ruuth}{Spatially partitioned embedded Runge--Kutta methods}

\section{Introduction}
The method-of-lines is a popular discretization technique for the numerical
solution of time-dependent partial differential equations.  In it, a spatial discretization is
applied to the PDE, yielding an initial value problem consisting of a
large system of ordinary differential equations (ODEs).
These are evolved by some time-stepping method, for example,
a Runge--Kutta method.  

The choice of a suitable time-stepping scheme may depend on a variety of considerations.
Typically, schemes are chosen to give good linear stability and accuracy, although in some applications
it is also useful or necessary to preserve the monotonicity or other properties of the true PDE solution.
Often, the key properties for determining a suitable time-stepping scheme vary locally in space.
Examples of such properties include the grid spacing, the coefficients of the PDE, 
the smoothness of the solution and the geometry of the domain.
As a consequence, it is possible that a scheme that is effective in one portion of the domain is
unsuitable or inefficient in another.  It is therefore natural to consider the development and analysis of
spatially partitioned time-stepping methods, in which different step sizes or different
methods are used over subdomains.

A class of methods that often benefit from spatially partitioned time-stepping
are PDE discretizations with grid adaptivity.  In regions where the solution is nonsmooth or exhibits rapid variation,
fine grids are needed; other regions may be more efficiently discretized with coarser meshes.
If all components of an ODE system are evolved using some 
explicit  time-stepping scheme and a single time step-size $\Dt$,  
the evolution of all components will be restricted by the stiffest components of the system.
Improved efficiency is often possible by considering 
time-stepping methods that vary their time step-size
according to the local mesh spacing.  Time-stepping 
schemes of this type are called {\it multirate} schemes.
The first multirate schemes for one-dimensional conservation laws were developed by Osher and Sanders in \cite{OsherSanders}.
Their approach carries out forward Euler time-stepping with step-sizes that vary locally.   
More recently higher-order methods have been considered;
for example, Tang and Warnecke \cite{tang2006high} develop second-order multirate schemes based on
Runge--Kutta or Lax-Wendroff type schemes.  
Other recent work includes that by Constantinescu and Sandu \cite{constantinescu2007multirate} where 
a simple construction algorithm is given 
to form a second-order accurate multirate scheme based on an arbitrary strong-stability-preserving (SSP) 
scheme of order two or higher.
Notably, their approach preserves a variety of monotonicity properties, such as positivity and
the maximum principle.   As explained by Hundsdorfer, Mozartova and Savcenco \cite{hundsdorfer2012monotonicity} multirate schemes 
for conservation laws are either locally inconsistent (e.g., \cite{OsherSanders}) or lack mass conservation (e.g., \cite{tang2006high}).  
Fortunately, order reduction due to such local inconsistency may be less severe
than expected, due to cancellation and damping effects; see 
\cite{hundsdorfer2012monotonicity} for details.   

Whereas multirate methods use different step sizes, in the present work we focus on
using different methods on different spatial subdomains.
Specifically, we investigate
spatially partitioned embedded Runge--Kutta (SPERK) 
schemes applied to one-dimensional PDEs such as conservation laws
\begin{align}
u_t + f(u)_x & = 0.
\end{align}
Here the term {\em embedded} refers to methods having the same coefficient
matrix $A$ (but different weights $b$), which
avoids the unnecessary duplication of computations that can occur when combining 
two unrelated time-stepping schemes.    Two  classes
of SPERK schemes are considered: {\em equation-based partitioning} and {\em flux-based partitioning}.
We shall find that equation-based partitioning maintains the overall accuracy of the schemes 
composing the embedded pair.  This approach is, however, not conservative and can lead to
wrong shock propagation speeds when applied to hyperbolic PDEs.   Flux-based partitioning
is conservative, but experiences a theoretical order reduction in the local consistency.  
In practice, however, we shall find that the overall accuracy of the combined scheme is unaffected.

The idea of adapting the time discretization locally in space is not new.  
Efficient combinations of implicit and explicit time discretizations have been
developed for problems where very large wave speeds appear in a localized regions, such as
oil reservoir modeling or fluid-structure interaction \cite{belytschko1979mixed,de2009error,timofeev2012application}.  Combinations of
higher- and lower-order implicit schemes have been used in order to preserve monotonicity
under large step sizes \cite{forth2001,duraisamy2003,duraisamy2007,Slingerland}.
Combinations of explicit methods have received less attention;
some techniques for simultaneous local adaptation of both temporal and spatial
discretizations were proposed
in \cite{harabetian1993nonconservative}, in the context of shock-capturing methods for
conservation laws.  Here we provide a general framework for local adaptation of just
the time discretization, for explicit integration of general time-dependent PDEs.

We will see that the methods we study can be viewed in the
general framework of partitioned or additive Runge--Kutta methods.
However, such methods are usually applied in a way that applies
different numerical treatment to {\em different physical processes},
whereas the emphasis here is on different numerical treatments for
{\em different spatial domains}.  These approaches are based on 
different motivations, and some approaches that work well for
the former fail for the latter (see Example \ref{godunov-example}).

To illustrate the use of spatially adaptive time-stepping schemes, 
consider a convection-diffusion problem in which the Reynolds number varies
spatially, such that the system is dominated by convection on one subdomain and dominated by diffusion elsewhere.   
The third-order, three-stage SSP Runge--Kutta scheme of Shu and Osher (SSPRK(3,3)) \cite{Shu/Osher:1988} 
might be desired for the convection-dominated regions, 
while a second-order Runge--Kutta--Chebyshev method, e.g., \cite{verwer1996},
might be preferred in diffusion-dominated regions.   
Unfortunately, neither scheme is particularly attractive on its own since SSPRK(3,3) requires 
small time steps when applied to diffusive problems, while the
Runge--Kutta--Chebyshev method is unstable when applied to convective problems.
On the other hand, a combination that applies each time-stepping scheme
where it is best suited may provide better linear stability than either scheme on its own.
An example of this kind is explored in Section \ref{sec:advectiondiffusion}.

There are other situations where spatially partitioned
time-stepping schemes show a strong potential.   For example, consider the evolution
of a conservation law involving both shocks and smooth regions.  The preservation of 
monotonicity properties may be the most crucial property  near the shock, 
whereas high-order accuracy may be of primary interest in smooth regions. 
Use of a spatially partitioned time-stepping scheme opens the possibility of
simultaneously obtaining good accuracy and monotonicity in such problems.
An example of this kind is explored in Section~\ref{sec:WENO},
where SPERK schemes are applied to fifth-order weighted essentially
non-oscillatory (WENO) spatial discretizations \cite{Shu:WENO_SiamReview}.
The motivation here is that
fifth-order SSPRK methods are complicated by their use
of the downwind-biased operator \cite{Shu/Osher:1988, Shu:1988,
Spiteri/Ruuth:newclass, Spiteri/Ruuth:downwindbias} while 
monotonicity and the corresponding SSP property is 
likely only useful in the vicinity of non-smooth
features.  In these regions WENO discretizations are formally
third-order accurate \cite{Shu:WENO_SiamReview}.
Thus, in our approach, a fifth-order linearly
stable Runge--Kutta scheme is used in smooth regions while an embedded
third-order SSP Runge--Kutta scheme is
used near shocks or other discontinuities.

The paper unfolds as follows.   Section~\ref{sec:SPERK} introduces equation-based partitioning and examines its conservation properties.
This is followed by the introduction and analysis of flux-based partitioning.  
Errors and positivity properties for both classes of schemes are considered in this section.
A variety of generalizations are given in Section~\ref{sec:generalizations}.   Section~\ref{sec:advectiondiffusion} 
gives some examples of SPERK schemes, and considers applications of the methods to a spatially discretized advection-diffusion equation.
Section~\ref{sec:WENO} considers SPERK schemes in the context of WENO spatial discretizations and a nonlinear
partitioning step.  In our approach, the WENO weights are used to select between
an SSP Runge--Kutta scheme, and a high-order scheme chosen for its linear stability properties.
Finally, Section~\ref{sec:discussion} concludes with a discussion of some other application areas and some of our current research
directions.

\section{Spatially partitioned embedded Runge--Kutta methods} \label{sec:SPERK}
We begin by introducing equation- and flux-based SPERK schemes.  An analysis
of the accuracy, conservation and positivity properties
of such schemes is also provided in this section.

\subsection{Equation-based partitioning}

Consider a system of $N$ ordinary differential equations
\begin{align} \label{ode}
U'(t) = G(U),
\end{align}
typically arising as the spatial discretization of a PDE where each
component in the solution approximates, for example, point values of
the PDE solution $u_i(t) \approx u(x_i,t)$ at discrete points $x_i, 1\le i \le N$.
To apply an $s$-stage Runge--Kutta method, we first compute the
stage values
\begin{subequations} \label{EPRK}
\begin{align} \label{RK-stages}
\s{j}     &= U^n + \Dt \sum_{k=1}^{s} a_{jk} G(\s{k}), ~~j=1,\ldots,s.
\end{align}
A standard Runge--Kutta method would then advance by one step using the formula
$U^{n+1} = U^n + \Dt \sum_{j=1}^s b_{j}  G(\s{j})$.
Instead, let a different set of weights be applied at each point
$x_i$, $1\le i \le N$, choosing between coefficients $b$ or $\hat{b}$.
This results in
\begin{align} \label{EPRK-b}
u_i^{n+1} &= u_i^n + \Dt \left[ \chi_i^n \sum_{j=1}^{s} b_{j}  g_i(\s{j})
+ (1-\chi_i^n) \sum_{j=1}^{s} \hat{b}_{j}  g_i(\s{j}) \right]
\end{align}
where $g_i$ is the $i$th component of $G$ and 
\begin{align}
  \chi^n_i = 
  \begin{cases}
    1,  &\text{if weights $b$ are used to compute $u^{n+1}_i$,}\\
    0,  &\text{if weights $\hat{b}$ are used to compute $u^{n+1}_i$.}
  \end{cases}
\end{align}
\end{subequations}
We write the coefficients of this method $A$, $b$, $\hat{b}$ using 
the same tableau notation that is
employed for embedded Runge--Kutta methods \cite{Hairer:ODEs1}:
\begin{align}  \label{eq:embedRK}
  \begin{array}{c|c}
    c & A \\
    \hline
    & \rule{0pt}{1.05em}b^T \\
    & \hat{b}^T
  \end{array}.
\end{align}
We refer to the embedded method \eqref{EPRK} as a \emph{spatially partitioned}
time-stepping method because the ``mask'' $\chi$ selects
which scheme to propagate at each point in space, at each time step.
In some cases, $\chi$ depends on the solution
values, although we do not explicitly represent this
for notational clarity.

\subsubsection{Connection to partitioned Runge--Kutta methods}

Methods of the form \eqref{EPRK} form a special subclass of
partitioned Runge--Kutta methods \cite{Hairer:ODEs1}.
Generally, methods in a partitioned RK pair may have different
coefficient matrices $A$ as well as different weights $b$.
In our embedded approach, methods have fewer
degrees of freedom available for their design, but they possess the
advantage of automatically satisfying the ``extra'' order conditions
for partitioned RK methods.
That is, if each of the component methods $(A,b)$ and $(A,\hat{b})$ is
accurate to order $p$, the SPERK method is also accurate to order $p$
in time
(see also Proposition~\ref{orderprop} below).

Method~\eqref{EPRK} partitions \eqref{ode} by equation; we refer to this type of partitioning as 
{\it equation-based partitioning}.
Because the $i$th equation corresponds to grid node $x_i$, $1\le i \le N$, our approach also gives
a spatial partitioning of a semi-discretized PDE.  
It is worth noting, however, that  equation-based partitioning is a very general technique that
does not require any correspondence to  grid locations.

\subsubsection{Conservation}
Many important physical phenomena are modeled by conservation laws, which in
one dimension have the form
\begin{align} \label{conslaw}
u_t + f(u)_x & = 0.
\end{align}
In the numerical solution of \eqref{conslaw}, one should use a conservative
scheme in order to ensure that shocks propagate at the correct speed.
Typically, \eqref{conslaw} is semi-discretized using a flux-differencing method:
\begin{align}  \label{eq:fluxdiff}
  u_i'(t) = - \frac{1}{\Dx} \left(f_\iph-f_\imh\right) 
\end{align}
where $f_\ipmh$ are numerical approximations to the flux at $x_\ipmh, 1\le i \le N$.
Integrating with a Runge--Kutta method gives
\begin{align*}
  u_i^{n+1} = u_i^n - \frac{\Dt}{\Dx} \left(\sum_{j=1}^s b_j f_\iph(\s{j})- \sum_{j=1}^s b_j f_\imh(\s{j})\right)
\end{align*}
where the stage values $\s{j}, 1\le j \le s,$ are defined in equation \eqref{RK-stages} above.
This method is conservative since corresponding fluxes cancel out (except at the boundaries) 
if we sum over the spatial index $i$, $1\le i \le N$.

Applying method \eqref{EPRK} gives instead
\begin{align*}
u_i^{n+1} = u_i^n - \frac{\Dt}{\Dx} &\left( \sum_{j=1}^s (\chi^n_i b_j + (1-\chi^n_i)\hat{b}_j) f_\iph(\s{j})\right. \\
                                &- \left. \sum_{j=1}^s (\chi^n_i b_j + (1-\chi^n_i)\hat{b}_j) f_\imh(\s{j})\right).
\end{align*}
This method is not conservative, since the flux terms at $x_\imh$ will not cancel
if $\chi^n_{i-1}\ne \chi^n_i$.  This can lead to solutions
in which discontinuities move at incorrect speeds.  Here we give an
example.

\begin{example}
We consider Burgers' equation, $u_t + \left(\frac{1}{2}u^2\right)_x=0$,
with the step function initial condition
\begin{align*}
u(x,t=0) & = \begin{cases}
2 & x\le 0, \\
0 & x>0.
\end{cases}
\end{align*}
We discretize in space with a finite difference flux-differencing scheme
using fifth-order WENO interpolation.  In time, we discretize using the SPERK scheme in
Table~\ref{tab:rk75ssp53_scheme} and equation-based partitioning.  We use a time step-size
of $\Dt = 0.6 \Dx$, corresponding to a CFL number of 1.2.
We take
\begin{align*}
\chi(x, t) & = \begin{cases}
0 & 0.01 < u(x,t) < 1.99, \\
1 & \text{otherwise.}
\end{cases}
\end{align*}
This simple choice of $\chi$ ensures that the SSP method is used
near the shock while the non-SSP method is used elsewhere.  It also (unfortunately)
ensures that the jumps in $\chi$ are near the shock, maximizing the effect of conservation errors.
The true shock velocity is $1$, but the numerical shock velocity converges
to approximately $0.925$.
If we replace $\chi$ above by $1-\chi$, the shock moves instead too rapidly. \qed
\end{example}

\subsection{Flux-based partitioning}

In \eqref{EPRK} we applied one of the RK schemes to each equation in the ODE system.
In order to obtain a conservative method,
we partition by the fluxes $f_{\iph}, 0\le i \le N$, rather than by the
equations.

Suppose we are given two Runge--Kutta methods
with identical coefficient matrices $A$, but different weights $b$ and $\hat{b}$.
Similar to equation-based partitioning, we first compute stage values 
according to \eqref{RK-stages},
which when applied to the flux-differencing semi-discretization \eqref{eq:fluxdiff} is
\begin{subequations} \label{flux-based}
\begin{align} \label{flux-based-RKstages}
y_i^{(j)}  &= u_i^n - \frac{\Dt}{\Dx} \sum_{k=1}^{s} a_{jk} \left(f_\iph(\s{j}) - f_\imh(\s{j})\right), ~~j=1,\ldots,s.
\end{align}
However, instead of varying the time-stepping method by equation, we
vary the method by flux.  An application of this partitioning to
\eqref{eq:fluxdiff} yields
\begin{equation}
\begin{split}
    u_i^{n+1} = u_i^n - \frac{\Dt}{\Dx}
        &\left(\chi^n_\iph \sum_{j=1}^s b_j f_\iph(\s{j}) + (1-\chi^n_\iph) \sum_{j=1}^s \hat{b}_j f_\iph(\s{j})\right. \\
            &- \left. \chi^n_\imh \sum_{j=1}^s b_j f_\imh(\s{j}) - (1-\chi^n_\imh) \sum_{j=1}^s \hat{b}_j f_\imh(\s{j})\right),
\end{split}
\end{equation}
where we have partitioned using the characteristic functions
corresponding to \emph{cell edges} $x_\iph$ (rather than grid points $x_i$)
\begin{align}
  \chi^n_\iph =
  \begin{cases}
    1,  &\text{if weights $b$ are used for fluxes at $x_\iph$,}\\
    0,  &\text{if weights $\hat{b}$ are used fluxes at $x_\iph$.}
  \end{cases}
\end{align}
\end{subequations}
Fig.~\ref{fig:pseudocode2} compares a pseudo-code implementation of
the equation-based \eqref{EPRK} and flux-based partitioning
\eqref{flux-based} for a three-stage SPERK scheme.

As we show next, the flux-based method corresponds to an additive
Runge--Kutta method instead of a partitioned Runge--Kutta method.

\subsubsection{Connection to additive Runge--Kutta schemes}
Consider again the flux-differencing formula \eqref{eq:fluxdiff}.
Let $\Phi(U)$ denote the vector of fluxes, with components $\phi_i(U) = f_{i+\hf}(U), 0 \le i \le N$.
We can write the flux-differencing method \eqref{eq:fluxdiff} in vector form:
\begin{align} \label{fluxdiff-phi-new}
  U'(t) = - \frac{1}{\Dx} \m{D} \Phi,
\end{align}
where $\m{D}$ is the $N\times (N+1)$  differencing matrix with $1$ on the
superdiagonal and $-1$ on the main diagonal.

We split the flux vector based on the vector $\chi \equiv [\chi_\frac{1}{2}, \chi_\frac{3}{2}, \ldots, \chi_{N+\frac{1}{2}}]^T$:
\begin{align}
  U'(t) = - \frac{1}{\Dx} \m{D} \Phi(U)
  &= - \frac{1}{\Dx} \m{D} \big[ \diag(\chi) \Phi(U) + (I-\diag(\chi)) \Phi(U)\big], \nonumber\\
  &= \underbrace{-\frac{1}{\Dx} \m{D} \diag(\chi) \Phi(U)}_{G_\chi(U)}
  \underbrace{-\frac{1}{\Dx} \m{D} (I-\diag(\chi)) \Phi(U)}_{G_{1-\chi}(U)}, \nonumber\\
  &= G_\chi(U) + G_{1-\chi}(U), \label{eq:splitd-new}
\end{align}
where $G_\chi(U) = {-\frac{1}{\Dx} \m{D} \diag(\chi) \Phi(U)}$,
$G_{1-\chi}(U) = {-\frac{1}{\Dx} \m{D} (I-\diag(\chi)) \Phi(U)}$
and
$\diag(\chi)$ is the $(N+1)\times (N+1)$ diagonal matrix that has $\chi$ as its main diagonal.
Flux-based partitioning applies different Runge--Kutta methods to $G_\chi$ and $G_{1-\chi}$.
Because we have chosen embedded pairs with identical
coefficient matrices $A$, the stage values $\s{j}, 1\le j \le s,$ are computed according
to \eqref{RK-stages} in the standard fashion.  Different weights for
$G_\chi$ and $G_{1-\chi}$ are applied, however, according to the formula
\begin{equation}  \label{EARK-new}
  U^{n+1} = U^{n} + \Dt \left[
    \sum_{j=1}^s b_j G_\chi(\s{j})  +
    \sum_{j=1}^s \hat{b}_j G_{1-\chi}(\s{j})
  \right].
\end{equation}

This approach  (i.e., using \eqref{EARK-new}) may be viewed as
applying an additive Runge--Kutta method \cite{Kennedy/Carpenter:2003}
to \eqref{eq:splitd-new}.
We emphasize that the methods we consider have identical $A$ matrices and therefore
form a special {\em embedded} subclass of the additive RK methods.
The effect of the splitting introduced in \eqref{eq:splitd-new} is crucial for
understanding the order of accuracy of  flux-based
SPERK schemes.
Theoretically, we might expect a reduction of order by one, a
result we establish in Theorem~\ref{thm:mixtimeOrder:flux-based}.
In practice, however, we observe that 
the order of accuracy of SPERK schemes is the minimum order of the two schemes for
both equation- and flux-based partitioning.
Section~\ref{sec:WENO} gives some numerical experiments illustrating
this property.

\begin{figure}
\newcommand{\ys}[2]{\texttt{y}^{(#1)}_{#2}}
\newcommand{\YS}[2]{\texttt{Y}^{(#1)}_{#2}}
\begin{lstlisting}[label=equation-based2,mathescape]
# Program 1: Equation-based SPERK implementation
# First compute the standard stage values
$\YS{1}{}$ = U
for $i$ in range(1, N):
   $\ys{2}{i}$ = u$_i$ - dt/dx$\big(a_{21}(\ys{1}{i}$-$\ys{1}{i-1})\big)$
for $i$ in range(1, N):
   $\ys{3}{i}$ = u$_i$ - dt/dx$\big(a_{31}(\ys{1}{i}$-$\ys{1}{i-1})$  +  $a_{32}(\ys{2}{i}$-$\ys{2}{i-1})\big)$
# Advance in time from u to unew using:
for $i$ in range(1, N):
   unew$_i$ = u$_i$ - dt/dx$\Big($ $\chi_i\!$ $\big(b_1(\ys{1}{i}$-$\ys{1}{i-1})$ + $b_2(\ys{2}{i}$-$\ys{2}{i-1})$ + $b_3(\ys{3}{i}$-$\ys{3}{i-1})\big)$
                 + (1-$\chi_i$)$\big(\hat{b}_1(\ys{1}{i}$-$\ys{1}{i-1})$ + $\hat{b}_2(\ys{2}{i}$-$\ys{2}{i-1})$ + $\hat{b}_3(\ys{3}{i}$-$\ys{3}{i-1})\big)\Big)$
\end{lstlisting}
\begin{lstlisting}[label=flux-based2,mathescape]
# Program 2: Flux-based SPERK implementation
# First compute the standard stage values as in lines 3-7 above
...
# Then advance in time from u to unew using:
for $i$ in range(1, N):
   unew$_i$ = u$_i$ - dt/dx$\Big($
         $\,\,\chi_\iph\,\,\,$$\big(b_1\ys{1}{i}$ + $b_2\ys{2}{i}$ + $b_3\ys{3}{i}$$\big)$  $\phantom{..}$    -      $\:\:\:\:\, \chi_\imh \,\,$$\big(b_1\ys{1}{i-1}$ + $b_2\ys{2}{i-1}$ + $b_3\ys{3}{i-1}$$\big)$
    + $\big(1-\chi_\iph\big)$$\big(\hat{b}_1\ys{1}{i}$ + $\hat{b}_2\ys{2}{i}$ + $\hat{b}_3\ys{3}{i}$$\big)$ - $\big(1-\chi_\imh\big)$$\big(\hat{b}_1\ys{1}{i-1}$ + $\hat{b}_2\ys{2}{i-1}$ + $\hat{b}_3\ys{3}{i-1}\big)\Big)$
\end{lstlisting}
\caption{Pseudocode implementation of the equation-based and flux-based
  SPERK algorithms.
  Here we apply one step of an explicit three-stage method to the first-order upwinding
  spatial discretization of the advection PDE $u_t + u_x = 0$.
  }
\label{fig:pseudocode2}
\end{figure}

\subsection{Accuracy} \label{sec:accuracy}
In general, additive and partitioned Runge--Kutta methods may exhibit lower
order convergence rates than either of their component methods.  In order for the method
to be fully accurate, additional order conditions relating the coefficients of the
different component methods must be satisfied.  Even then, in our context of spatial
splitting, such methods can exhibit order reduction when the coefficient matrices $A$ differ
\cite{hundsdorfer2012monotonicity}.

\begin{example} \label{godunov-example}
Flux-based partitioning based on a characteristic function $\chi$ gives the ODE system~(\ref{eq:splitd-new}).
The classical Godunov operator splitting is a first-order method that may be
applied to this system.  Applying this splitting yields
\begin{align*}
U^* & = U^n + \Dt G_\chi(U^n), \\
U^{n+1} & = U^* + \Dt G_{1-\chi}(U^*).
\end{align*}
This can be written as an additive Runge--Kutta method with coefficient arrays
\begin{align*}
  \begin{array}{c|cc}
    0 & 0 & 0 \\
    1 & 1 & 0 \\ \hline
      & 1 & 0 \\
  \end{array}, & &
  \begin{array}{c|cc}
    0 & 0 & 0 \\
    0 & 0 & 0 \\ \hline
      & 0 & 1 \\
  \end{array}.
\end{align*}
Unlike all other pairs considered in this paper, this pair uses different 
coefficients $a_{ij}$.

Consider the advection equation $u_t + u_x = 0$ discretized in space
by upwind differencing:
\begin{align} \label{updiff}
u_i'(t) = g_i(U)
\end{align}
where $g_i(U) = -\frac{1}{\Dx} \left(u_{i} - u_{i-1}\right)$.  Let $G=[g_i(U)]$ and
take $\Dt=\Dx$ and constant initial data, $u^0_i=1, -\infty < i < \infty$.  We partition the
domain into two regions by choosing an integer $J$ with
$\chi_\imh = 1$ for $i \le J$ and $\chi_\imh = 0$ for $i > J$.
After one step with this method, the computed solution
is
\begin{align*}
u^1_i &=
\begin{cases}
    1 & \text{ for } i < J, \\
    0 & \text{ for } i = J, \\
    2 & \text{ for } i = J+1, \\
    1 & \text{ for } i>J+1.
\end{cases}
\end{align*}
Despite the constant initial data, the local error is $\Oop(1)$ in the maximum norm.
\qed
\end{example}

The difficulty encountered in the example above is typical of methods
in which $A\ne \hat{A}$.  In the next two theorems, we establish the
accuracy of equation- and flux-based SPERK methods.

\begin{theorem}\label{thm:mixtimeOrder}
Suppose that the partitioned Runge--Kutta method \eqref{EPRK} with coefficients $(A,b,\hat{b})$
is applied to the semi-discretization \eqref{ode}.
Then the fully discretized system has a local order of accuracy
equal to $\min(p,\hat{p})$ where $p$ (respectively, $\hat{p}$) is the order of accuracy 
of the full discretization obtained by using the Runge--Kutta method with coefficients $(A,b)$ 
(respectively, $(A,\hat{b})$).
\end{theorem}
\begin{proof}
The full discretizations are one-step methods of the form 
\begin{align*}
w^{n+1}_i & = Q_i(W^n), \\
\widehat{w}^{n+1}_i & = \widehat{Q}_i(\widehat{W}^n).
\end{align*}
The exact solution satisfies
\begin{subequations}\label{trunc}
\begin{align} 
u(x_i,t_{n+1}) & = Q_i(u_h(t_n)) + \Oop(\Dt^{p+1}), \\
u(x_i,t_{n+1}) & = \widehat{Q}_i(u_h(t_n)) + \Oop(\Dt^{\hat{p}+1}).
\end{align}
\end{subequations}
where $u_h(t_n)=[u(x_1,t_n), u(x_2,t_n),\ldots,u(x_N,t_n)]^T$ is a vector of true solution
values at the grid nodes at time $t_n$.
Here we have assumed that $\Dx$ is given by some prescribed relationship in terms
of $\Dt$, so that the error can be characterized in terms of $\Dt$ only.

We now determine the local truncation error of the discretization:
starting from the exact solution at time $t_n$,
the solution computed by the partitioned method is
\begin{align*}
    u^{n+1}_i & = \chi_i^n Q_i(u_h(t_n)) + (1-\chi_i^n) \widehat{Q}_i(u_h(t_n)).
\end{align*}
Applying \eqref{trunc}, we find
\begin{align*}
u^{n+1}_i & =  \chi_i^n \cdot u(x_i,t_{n+1}) + \Oop(\Dt^{p+1}) + (1-\chi_i^n)\cdot u(x_i,t_{n+1}) + \Oop(\Dt^{\hat{p}+1}), \\
                 & = u(x_i,t_{n+1}) + \Oop(\Dt^{\min(p,\hat{p})+1}).
\end{align*}
\end{proof}

In order to prove accuracy of the flux-based decomposition approach, we need
to know how accurately a Runge--Kutta method approximates the fluxes.
\begin{lemma} \label{w-order-lemma}
Suppose we are given an initial value problem \eqref{ode} for $U\in\R^n$, a Runge--Kutta method $(A,b)$ of order $p$,
and a smooth function $V : \R^n \to \R^n$. Let $W(t) = \int_0^t V(U(s))ds$,
$W^n = W(t_n)$ and compute $W^{n+1}$ using \eqref{RK-stages} and
\begin{align} \label{w-scheme}
W^{n+1} = W^n + \Dt \sum_{j=1}^s b_j V(\s{j}).
\end{align}
This scheme approximates $W(t)$ to order $p$, i.e.
$W^{n+1} = W(t_{n+1}) + \Oop(\Dt^{p+1})$.
\end{lemma}
\begin{proof}
Suppose an $\Oop(\Dt^p)$ Runge--Kutta method is applied to the system
\begin{align*}
U' & = G(U), \\
W' & = V(U).
\end{align*}
Stage values are given by
\begin{align*} 
\s{j}    & =  U^n + \Dt \sum_{k=1}^s a_{jk} G(\s{k}) \\
Z^{(j)}  & =  W^n + \Dt \sum_{k=1}^s a_{jk} V(\s{k}), j=1,\ldots,s.
\end{align*}
The $Z^{(j)},  j=1,\ldots,s,$ are not used and would not be computed in practice.
We advance by one step via
\begin{align*}
U^{n+1} & = U^n + \Dt \sum_{j=1}^s b_j G(\s{j}),  \\
W^{n+1} & = W^n + \Dt \sum_{j=1}^s b_j V(\s{j}).
\end{align*}
Recall that the Runge--Kutta method gives an error that is $\Oop(\Dt^p)$.  Applying this fact
to the second equation gives the desired result.
\end{proof}

From Lemma \ref{w-order-lemma} we obtain:
\begin{lemma} \label{pairs-small-diff}
Let $V : \R^n \to \R^n$ be a smooth function, and let $(A,b),(A,\hat{b})$ be an
embedded RK pair with order $p,\hat{p}$.
Given an initial value problem \eqref{ode} in $U$, let $\s{j}$ denote the stage
values in \eqref{RK-stages}; then
\begin{align} \label{V_acc}
\sum_{j=1}^s b_j V(\s{j}) = \sum_{j=1}^s \hat{b}_j V(\s{j}) + \Oop(\Dt^{\min(p,\hat{p})}).
\end{align}
\end{lemma}
\begin{proof}
Applying Lemma \ref{w-order-lemma} to both schemes gives
\begin{align*}
    W(t) + \Dt \sum_{j=1}^s       b_j V(\s{j}) = W(t+\Dt) + \Oop(\Dt^{p+1}), \\
    W(t) + \Dt \sum_{j=1}^s \hat{b}_j V(\s{j}) = W(t+\Dt) + \Oop(\Dt^{\hat{p}+1}).
\end{align*}
Combining these gives the stated result.
\end{proof}

One must take care in applying Lemma \ref{pairs-small-diff} to a PDE semi-discretization,
since the error constant appearing in \eqref{V_acc} might involve a factor like
$\Dx^{-r}$, $r>0$, which would grow as the spatial grid is refined.
A detailed analysis of the propagation of the spatial discretization
error within a Runge--Kutta step is beyond the scope of this work;
we refer the interested reader to
\cite{Sanzserna:1986:RKorderreduction,Carpenter/Gottlieb/Don:1995:RKorder}.
In the following theorem we simply
assume that the error constant in \eqref{V_acc} is bounded as $\Dx \to 0$.
Note that this is an assumption on each of the component RK schemes separately.
Theorem \ref{thm:mixtimeOrder:flux-based} indicates that if each of the component schemes gives a solution
free from order reduction (in the sense just described), then the embedded pair
loses at most one order of accuracy.

\begin{theorem}\label{thm:mixtimeOrder:flux-based}
Suppose that the flux-based spatially partitioned Runge--Kutta method \eqref{flux-based} with coefficients $(A,b,\hat{b})$
is applied to the semi-discretization \eqref{eq:fluxdiff} with $\Dt=\Oop(\Dx)$.  Furthermore,
suppose that the error constant appearing in \eqref{V_acc} is bounded as $\Dx \to 0$ when
one takes $V=f_\ipmh$.
Then the full discretization of \eqref{eq:fluxdiff}  has order of accuracy equal to $\min(p,\hat{p})-1$
where $p$ (respectively, $\hat{p}$) is the order of accuracy of the full discretization obtained
by using the RK method with coefficients $(A,b)$ (respectively, $(A,\hat{b})$).
\end{theorem}
\begin{proof}
The full discretizations are
\begin{align*}
    w^{n+1}_i & = w^n_i - \frac{\Dt}{\Dx}\left(\sum_{j=1}^s b_j f_\iph(\s{j}) - \sum_{j=1}^s b_j f_\imh(\s{j})\right), \\
    \hat{w}^{n+1}_i & = \hat{w}^n_i - \frac{\Dt}{\Dx}\left(\sum_{j=1}^s \hat{b}_j f_\iph(\s{j}) - \sum_{j=1}^s \hat{b}_j f_\imh(\s{j})\right).
\end{align*}
Using Lemma~\ref{pairs-small-diff} with $V(u) = f_\ipmh(u)$ gives
\begin{align} \label{flux-accuracy}
  \sum_{j=1}^s b_j f_\ipmh(\s{j}) = \sum_{j=1}^s \hat{b}_j f_\ipmh(\s{j}) + \Oop(\Dt^{\min(p,\hat{p})}).
\end{align}
Flux-based partitioning gives
\begin{align*}
    u_i^{n+1} & = u_i^n - \frac{\Dt}{\Dx}
        \Bigg(\chi^n_\iph \sum_{j=1}^s b_j f_\iph(\s{j}) + (1-\chi^n_\iph) \sum_{j=1}^s \hat{b}_j f_\iph(\s{j}) \\
        &\qquad\qquad\qquad
            - \chi^n_\imh \sum_{j=1}^s b_j f_\imh(\s{j}) - (1-\chi^n_\imh) \sum_{j=1}^s \hat{b}_j f_\imh(\s{j})\Bigg) \\
            & = u_i^n - \frac{\Dt}{\Dx}
        \left(\sum_{j=1}^s b_j f_\iph(\s{j}) - \sum_{j=1}^s b_j f_\imh(\s{j})\right) + \Oop(\Dt^{\min(p,\hat{p})+1}/\Dx),
\end{align*}
where we have used \eqref{flux-accuracy}.
The first two terms on the right are just the solution computed by the scheme
with weights $b_j$, which is accurate to order $p$ by assumption.
Thus, inserting the exact solution in the last equation above
and using the assumption that $\dt = \Oop(\dx)$
gives the desired result.
\end{proof}
\begin{remark}
More generally, for a semi-discretization of an evolution PDE that is order $q$
in space, one takes (for explicit methods) $\Dt = \Oop(\Dx^q)$ and a factor
of $\Dx^{-q}$ appears in the spatial discretization.  The net effect is the
same in that it may reduce the accuracy by a factor of $\Dt$.
\end{remark}

\begin{remark}
It might be tempting to try to apply the theorem to the general
additive Runge--Kutta method
\begin{align}\label{eq:general_ark}
  \begin{array}{c|c}
    c & A \\
    \hline
    & \rule{0pt}{1.05em} b^T
  \end{array}
  \qquad\text{with}\qquad
  \begin{array}{c|c}
    \hat{c} & \hat{A} \\
    \hline
    & \rule{0pt}{1.05em} \hat{b}^T
  \end{array},
\end{align}
by first making a larger additive Runge--Kutta method that \emph{does} share
abscissae and stage weights, namely
\begin{align}\label{eq:bigembed}
  \begin{array}{c|ccc}
    c & A & 0\\
    \hat{c} & 0 & \hat{A} \\
    \hline
    & \rule{0pt}{1.05em} b^T & 0
  \end{array}
  \qquad\text{with}\qquad
  \begin{array}{c|ccc}
    c & A & 0\\
    \hat{c} & 0 & \hat{A} \\
    \hline
    & 0 &  \rule{0pt}{1.05em} \hat{b}^T
  \end{array},
\end{align}
and then concluding that Theorem~\ref{thm:mixtimeOrder} applied to the latter
implies the results on the former.
Indeed the theorem does apply to \eqref{eq:bigembed}; however,
the implication does not follow because \eqref{eq:general_ark} and
\eqref{eq:bigembed} are \emph{not the same method}.
To see this, note that \eqref{eq:general_ark} computes stages combining function values
from the two schemes (according to $\chi$) whereas
\eqref{eq:bigembed} computes each stage value using only one scheme.
\end{remark}

\subsection{Positivity} \label{sec:positivity}
The nonlinear stability properties of SPERK schemes are of interest as well.
In this section we consider the positivity of flux-based SPERK schemes
that are comprised of two explicit SSP
Runge--Kutta schemes.

We begin by reviewing some results.   Of particular relevance to our derivations is that
SSP Runge--Kutta schemes may be re-written in optimal Shu--Osher form \cite{Shu:1988}
via the transformation provided in \cite{Ferracina/Spijker:SO05}.
For the ODE system $U'=G(U)$ this yields schemes of the form
\begin{subequations}
\label{eq:alpha_beta_rk}
\begin{align}
\s{j} & = \left(1-\sum_{k=1}^{j-1}\alpha_{jk} \right) U^n + \sum_{k=1}^{j-1} (\alpha_{jk} \s{k} +
              \Delta t \beta_{jk} G(\s{k})), \\
U^{n+1} & = \s{s+1}
\end{align}
\end{subequations}
where all  $\alpha_{jk} \geq 0,$ $\sum_{k=1}^{j-1}\alpha_{jk} \le 1$ and $j=1,2,\ldots,s+1$.

If both sets of coefficients $\alpha_{jk},
\beta_{jk}$ are nonnegative, 
and forward Euler is
positivity-preserving, then it may be shown that
the Runge--Kutta scheme preserves positivity under
a suitable time step restriction
\cite{Shu/Osher:1988,Gottlieb/Shu/Tadmor:2001}:

\begin{lemma}
\label{thm:cfl_pos}
If the forward Euler method is positivity-preserving under the 
restriction $0 \le \Delta t \le \Delta t_{FE}$, then the Runge--Kutta
method~(\ref{eq:alpha_beta_rk}) with $\beta_{jk}\ge 0$ is positivity-preserving provided
\[
\Delta t \le \sspcoeff \Delta t_{FE},
\]
where $\sspcoeff$ is the SSP coefficient
\[
\sspcoeff \equiv \min \left\{c_{jk}: 1\le k < j \le s+1\right\} \mbox{~where~} c_{jk}=\left\{\begin{array}{ll}
\frac{\alpha_{jk}}{\beta_{jk}} & \mbox{if $\beta_{jk} \neq 0,$ } \\
\infty & \mbox{otherwise.}
\end{array}
\right.
\]
\end{lemma}
We now show the corresponding result for flux-based SPERK schemes.
\begin{theorem}
Suppose that the flux-differencing semi-discretization \eqref{eq:fluxdiff}
of the one-dimensional conservation law~(\ref{conslaw}) satisfies the positivity
properties
\begin{align}
w_i -  \frac{\delta t_1}{\Dx} f_\iph(W) & \ge 0, \: \mbox{for all $\delta t_1 \le \Dt_1$}, \label{eq:hypoth1} \\
w_i +  \frac{\delta t_2}{\Dx} f_\imh(W) & \ge 0, \: \mbox{for all $\delta t_2 \le \Dt_2$}  \label{eq:hypoth2}
\end{align}
for all $w_i\ge0, 1\le i \le N$.  Further, suppose that
a flux-based SPERK scheme is applied to (\ref{eq:fluxdiff}), and
that the schemes composing the embedded pair have SSP coefficients
$\sspcoeff$ and $\hat{\sspcoeff}$.  Then the full discretization is 
positivity-preserving for time steps satisfying
\begin{equation}
0\le \Dt \le \min(\sspcoeff,\hat{\sspcoeff}) \Dt_* \label{eq:positivity}
\end{equation}
where 
\[
\Dt_* = \left\{ \begin{array}{ll}
\Dt_2 & \mbox{if $\Dt_1=\infty$},\\
\Dt_1 & \mbox{if $\Dt_2=\infty$},\\
\frac{\Dt_1 \Dt_2}{\Dt_1 + \Dt_2} & \mbox{otherwise.}
\end{array}
\right.
\]
\end{theorem}
\begin{proof}
We give a proof for the case where  $\Dt_1$ and $\Dt_2$ are finite; the proof
for the two remaining cases is straightforward and follows in a similar fashion.

Let $\chi_\iph^n$ and $1-\chi_\iph^n$ specify the (non-negative) weightings of the two schemes
at the cell boundaries at time step $n$.  The update for $U$ is given by
\begin{align}
u^{n+1}_i & =  u^n_i + \lambda \sum_{j=1}^{s}       b_i\left( -  \chi_\iph^n f_\iph^{(j)} +  \chi_\imh^n f_\imh^{(j)}\right)
   \nonumber \\  &\qquad\qquad\qquad\qquad
     +  \lambda \sum_{j=1}^{s} \hat{b}_i\left( -  (1-\chi_\iph^n) f_\iph^{(j)} +  (1-\chi_\imh^n) f_\imh^{(j)}\right) \nonumber \\
          & =  \chi_\iph^n \left(\gamma_p u^n_i + \lambda \sum_{j=1}^{s} - b_i f_\iph^{(j)}\right) +  (1-\chi_\iph^n) \left(\gamma_p u^n_i   + \lambda \sum_{j=1}^{s}       - \hat{b}_i f_\iph^{(j)}\right)  \nonumber \\
          &\,\,\, +  \chi_\imh^n \left(\gamma_m u^n_i + \lambda \sum_{j=1}^{s}  b_i f_\imh^{(j)}\right) +  (1-\chi_\imh^n) \left(\gamma_m u^n_i   + \lambda \sum_{j=1}^{s}        \hat{b}_i f_\imh^{(j)}\right) \label{eq:terms}
\end{align}
where $\lambda = \frac{\Dt}{\Delta x}$, $\gamma_p = \frac{\Dt_2}{\Dt_1+\Dt_2}$, and $\gamma_m = \frac{\Dt_1}{\Dt_1+\Dt_2}$.
We denote the fluxes associated with stage $j$ by $f_\ipmh^{(j)} \equiv f_\ipmh(\s{j})$ where $\s{j}$ is the
approximation of $u_h(t_n)$ at the $j$th stage.
The first term will be non-negative if
\begin{equation}
\gamma_p u^n_i + \lambda \sum_{j=1}^{s} - b_i f_\iph^{(j)} \label{eq:chkpos}
\end{equation}
is non-negative.  Transforming~(\ref{eq:chkpos}) to optimal Shu--Osher form yields
\[
\gamma_p \left(1-\sum_{j=1}^s \alpha_{s+1,j} \right)u^n_i +
          \gamma_p \left( \sum_{j=1}^{s} \alpha_{s+1,j} y_i^{(j)} - \frac{\lambda}{\gamma_p} \beta_{s+1,j} f_\iph^{(j)}\right).
\]
We observe that forward Euler is 
positivity-preserving when applied to the semi-discretization \eqref{eq:fluxdiff} provided $\Dt \le \Dt_*$.
An application of  Lemma~\ref{thm:cfl_pos} gives  $\s{j} \ge 0, 1\le j \le s$, 
provided the step-size restriction (\ref{eq:positivity}) is satisfied and $U^n \ge 0$, where 
vector inequalities are to be interpreted as being taken componentwise.
Combining this result with $\Dt \le \sspcoeff \Dt_*$ and hypothesis~(\ref{eq:hypoth1}) yields that the first term 
in (\ref{eq:terms}) is non-negative.  The proof follows by applying a similar analysis to
the second, third, and fourth terms in (\ref{eq:terms}).
\end{proof}

The theorem may be applied to a variety of common discretizations.  For example, consider
an upwind discretization of the linear advection equation
\[
u_t + a u_x = 0, \quad a>0,
\]
on a grid with uniform spacing $\Delta x$. In this example $f_\iph = a u_i$  from which we observe that
 $\Dt_2 = \infty$ and $\Dt_1=\frac{1}{a}\Delta x$.
Assuming that the schemes forming the embedded pair
have SSP coefficients $\sspcoeff$ and $\hat{\sspcoeff}$ we find that
the flux-based SPERK scheme is
positivity-preserving provided
\begin{equation}
\Delta t \le \frac{1}{a} \min(\sspcoeff,\hat{\sspcoeff}) \Delta x. 
\end{equation}

We can also apply the theorem to a discretization of the diffusion equation.  Consider
\[
u_t - \nu u_{xx} = 0, \nu>0,
\]
on a grid with uniform spacing $\Delta x$. Taking $f_\iph = \frac{\nu}{\Delta x}(-u_{i+1}+u_i)$ gives
$\Dt_1 = \Dt_2 = \frac{1}{\nu}\Delta x^2$.  Assuming that the schemes forming the embedded pair
have SSP coefficients $\sspcoeff$ and $\hat{\sspcoeff}$, we find that the flux-based SPERK scheme is
positivity-preserving under the time step-size restriction
\begin{equation}
\Delta t \le \frac{1}{2\nu} \min(\sspcoeff,\hat{\sspcoeff}) \Delta x^2. 
\end{equation}

The proof for equation-based partitioning follows in a similar manner.  For completeness we state the result below.
\begin{theorem}
Suppose that the semi-discretization (\ref{ode}) 
of the one-dimensional conservation law~(\ref{conslaw}) satisfies the positivity
property
\[
W + \Dt G(W) \ge 0
\]
for all $W\ge0$ and $\Dt \le \Dt_{FE}$ where the inequalities are taken component-wise.
Further, suppose that
an equation-based SPERK scheme is applied to (\ref{ode}), and
that the schemes composing the embedded pair have SSP coefficients
$\sspcoeff$ and $\hat{\sspcoeff}$.  Then the full discretization is 
positivity-preserving for time steps satisfying
\begin{equation*}
0\le \Dt \le \min(\sspcoeff,\hat{\sspcoeff}) \Dt_{FE}.
\end{equation*}
\end{theorem}

\section{Generalizations} \label{sec:generalizations}
In this section we provide some useful generalizations of our previous
spatial partitioning methods.  For notational simplicity we state these results
for partitions that are invariant in time; however, it is also
possible to vary the partitioning function $\chi$ at each time
level $n$.

\subsection{More than two methods}
Both  equation- and flux-based partitioning can be generalized to more than two methods.
Accuracy and positivity
results essentially identical to those in Sections~\ref{sec:accuracy} and \ref{sec:positivity} 
can be shown for  both cases.

\subsubsection{Equation-based partitioning}
Let $(\Iint_1, \Iint_2,\dots \Iint_r)$ be
a partitioning of the integers from 1 to $N$ (i.e., the partitioning satisfies
$\cup_{k=1}^r \Iint_k=\{1,2,\dots,N\}$,
and $\Iint_j \cap \Iint_k = \emptyset$ if $j \neq k$).  
The generalization to $r$ embedded methods defined by the coefficients
$A, b^{(1)},\dots,b^{(r)}$ is defined as
\begin{subequations} \label{generalEPRK}
\begin{align}
\s{j}     & = U^n + \Dt \sum_{k=1}^s a_{jk} G(\s{k}), ~~1\le j \le s, \nonumber \\
U^{n+1} & = U^n + \Dt \sum_{k=1}^r \sum_{j=1}^s b_{j}^{(k)} \diag({\chi^{(k)}}) G(\s{j}), \nonumber
\end{align}
\end{subequations}
where 
\begin{align}
\chi^{(k)}_{i} = \begin{cases}
1 & \text{if } i \in \Iint_k, \nonumber \\
0 & \text{otherwise.}
\end{cases}
\end{align}

\subsubsection{Flux-based partitioning}
Flux-based partitioning can also be written more generally in terms of 
$b^{(1)},\dots,b^{(r)}$.  In this case we take
\begin{subequations}\label{eq:generalARK}
\begin{align}
\s{j}     & = U^n - \frac{\Dt}{\Dx} \sum_{k=1}^s a_{jk} \m{D} \Phi(\s{k}), 1\le j \le s, \nonumber \\
U^{n+1} & = U^n - \frac{\Dt}{\Dx} \sum_{k=1}^r \sum_{j=1}^s b_{j}^{(k)} \m{D} \diag({\chi^{(k)}}) \Phi(\s{j}), \nonumber
\end{align}
\end{subequations}
to combine the $r$ different embedded schemes.

\subsection{Blending and accuracy}
\label{sec:blending}
Rather than shifting discontinuously between methods from one spatial point to 
the next, an appealing approach is to have a transition region in which a weighted
average of the two methods is used, with the weight shifting from one method to
the other over the transition region.
The following proposition shows that this approach is still accurate
in time, which is a necessary condition for the full SPERK discretization
to be accurate.

\begin{proposition} \label{orderprop}
Let $(A,b^{(1)}), (A,b^{(2)}), \dots, (A,b^{(r)})$ denote a set of embedded RK methods, and let
$p_k$ denote the order of accuracy of method $(A,b^{(k)})$.  Then the RK method
$$\left(A, \sum_{k=1}^r \alpha_k b^{(k)} \right) \mbox{ where } \sum_{k=1}^r \alpha_k = 1$$
has order of accuracy at least $p=\min_k p_k$.
\end{proposition}
\begin{proof}
The order conditions involve only expressions that are linear in the weights.
Since all of the component method coefficients satisfy the conditions up to
order $p=\min_k p_k$, the averaged method does as well.
\end{proof}
\begin{remark}
The proposition applies only to embedded methods.
If the component methods had different coefficient matrices $A$, their average
would in general be only first-order accurate, even if they share common abscissas $c$.
\end{remark}

Theorems~\ref{thm:mixtimeOrder} and \ref{thm:mixtimeOrder:flux-based} for equation- and flux-based  SPERK methods
can be extended to the blended case in this way.  
In practice, we observe the full accuracy when blending using either
approach.

\section{Example: Advection-diffusion} \label{sec:advectiondiffusion}
The semi-discretization of certain time-dependent PDEs leads to systems of ODEs with eigenvalues
near the negative real axis, whereas the semi-discretization of others leads
to systems with eigenvalues near the imaginary axis.  
Appropriate time integrators must include the corresponding portions
of the complex plane in their absolute stability regions.
In problems with strongly varying coefficients or mesh spacing
the relevant portion of the complex plane may vary spatially.  This
may complicate the selection of an appropriate time integrator.

Consider, for example, the nonlinear advection-diffusion equation
\begin{align} \label{adv-diff}
u_t + (b(x) u)_x = (a(x) (u^2)_x)_x,
\end{align}
with periodic boundary conditions on $[0,1]$ and initial conditions 
\[
u(x, 0) = \tfrac{1}{10} \sin^3(2 \pi x) + 2.
\]
Discretizing the spatial derivatives in \eqref{adv-diff} with three-point
centered differences yields the flux-differencing method \eqref{eq:fluxdiff}
with
\[
f_{i+\frac{1}{2}} = b(x_{i+\frac{1}{2}})\left(\frac{u_{i+1}+u_{i}}{2}\right) - a(x_{i+\frac{1}{2}})  \frac{1}{\Dx} \left(u^2_{i+1}-u^2_{i}\right)
\]
and a constant mesh spacing $\Dx$.
We take $a(x)>0$ and $b(x)$ to be functions that are periodically defined by
\begin{align*}
a(x) &=  \tfrac{1}{1000} + \tfrac{1}{10000} (\cos(2\pi x-\tfrac{1}{2}\pi)+1)^{10}, \\
b(x) &=  1+\tfrac{1}{10}(\cos(2\pi x-\tfrac{3}{2}\pi)+1)^{10}.
\end{align*}
As shown in Fig.~\ref{fig:ab}, there
is a region centered at $x=0.25$ where diffusion dominates and
negative real axis inclusion is critical for the design of the time-stepping scheme.
Similarly, in the region around $x=0.75$, convection dominates and imaginary axis inclusion is
the most crucial design feature.
The strong spatial variation in the dominant term makes this a good problem 
for testing the performance of SPERK schemes.  

\begin{remark}  This nonlinear advection-diffusion problem serves as a prototype
where the relative importance of real and imaginary eigenvalues varies spatially.  
We now develop explicit time-stepping schemes suitable for such problems.
Explicit schemes have the advantage of being  easy to implement and
do not require the solution of nonlinear systems or the inversion of matrix systems.
We remark that for the case of advection-diffusion, 
implicit-explicit (IMEX) time-stepping schemes might also be considered.
IMEX schemes apply an explicit scheme to nonstiff terms and an implicit scheme to
stiff terms, and are particularly effective when a linear, symmetric diffusion term arises.
See, e.g., \cite{Ascher/Ruuth/Wetton:95imex} for further details.
\end{remark}
 
\begin{figure}
  \centerline{\includegraphics[height=20ex]{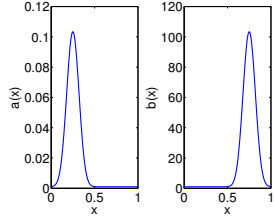}}
  \caption{The variable coefficients $a(x)$ (left) and $b(x)$ (right).   We observe a spatial variation in the relative importance of
advection and diffusion.}
  \label{fig:ab}
\end{figure}

\subsection{Second-order embedded pairs}
We now design three- and four-stage embedded explicit Runge--Kutta pairs
suitable for this problem.
In each embedded pair, one method has a stability polynomial that maximizes (or nearly maximizes)
the real axis interval of absolute stability, while the other method has a stability
polynomial that maximizes (or nearly maximizes) the imaginary axis interval of
absolute stability.
Given a coefficient matrix $A$, the stability polynomial of a Runge--Kutta method
(and hence its linear stability properties) can be chosen by solving a linear system
of equations for the weights $b_j$.
We do not consider two-stage pairs since all second-order two-stage schemes
have the same stability polynomial.

The stability polynomial of a three-stage, second-order RK method has the form
$R(z) = 1 + z + z^2/2 + \alpha_3 z^3$ where $\alpha_3 = b^T A c$.
For real-axis stability, we take the 3-stage Runge--Kutta--Chebyshev method of the
family presented in \cite{verwer1996}, which we refer to as RKC(3,2), as the second
scheme of our embedded pair.  The stability
polynomial of this method is a Bakker--Chebyshev polynomial and
contains nearly the largest possible portion of the negative real axis over the class
of three-stage second-order schemes.
The first scheme is obtained by setting $\alpha_3=1/4$, which approximately
maximizes imaginary axis inclusion, and selecting weights $\hat{b}_j$ corresponding
to the resulting polynomial.  This gives the following embedded pair:
\begin{equation}
  \begin{footnotesize}
  \begin{array}{c|ccc}
    0           &         &         &           \\
    3/8      &  3/8 &         &           \\
    3/8      &  3/16 &  3/16 &          \\
    \hline
    \rule{0pt}{1.05em}
    \hat{b}^T     &  -1/3 &  -20/9 &  32/9   \\
    b^{T}       &  -1/3    &  4/9    &   8/9
  \end{array}
  \end{footnotesize}
  \label{eq:RKC32}
\end{equation}
The absolute stability regions are shown in Fig.~\ref{fig:linear-stability} (left).

The stability polynomial of a four-stage, second-order RK method has the form
$R(z) = 1 + z + z^2/2 + \alpha_3 z^3 + \alpha_4 z^4$ where $\alpha_3 = b^T A c$
and $\alpha_4 = b^T A^2 c$.
The optimal imaginary axis inclusion occurs for the choice $\alpha_3=1/6$, $\alpha_4=1/24$,
which gives the stability polynomial of the
classical fourth-order Runge--Kutta method.  Hence we take the classical method,
which we shall denote by RK4, as the first
scheme of our pair.
Sufficient degrees of freedom remain to obtain a scheme with the same linear stability as
the optimal four-stage, second-order scheme RKC(4,2).  This gives the following embedded pair (with absolute stability regions as shown in Fig.~\ref{fig:linear-stability} right)
\begin{equation}
  \begin{footnotesize}
  \begin{array}{c|cccc}
    0           &         &         &  &         \\
    1/2      &  1/2 &         &     &      \\
    1/2      &  0   &  1/2 &       &   \\
    1      &   0  & 0  &  1     &   \\
    \hline
    \rule{0pt}{1.05em}\hat{b}^T     &  1/6 &  1/3 & 1/3 &  1/6  \\
    b^{T}       &  2/125    &  17/25  & 36/125  &   2/125
  \end{array}
  \end{footnotesize}
 \label{eq:RK4}
\end{equation}

\begin{figure}
  \centerline{%
    \includegraphics[width=0.4\textwidth]{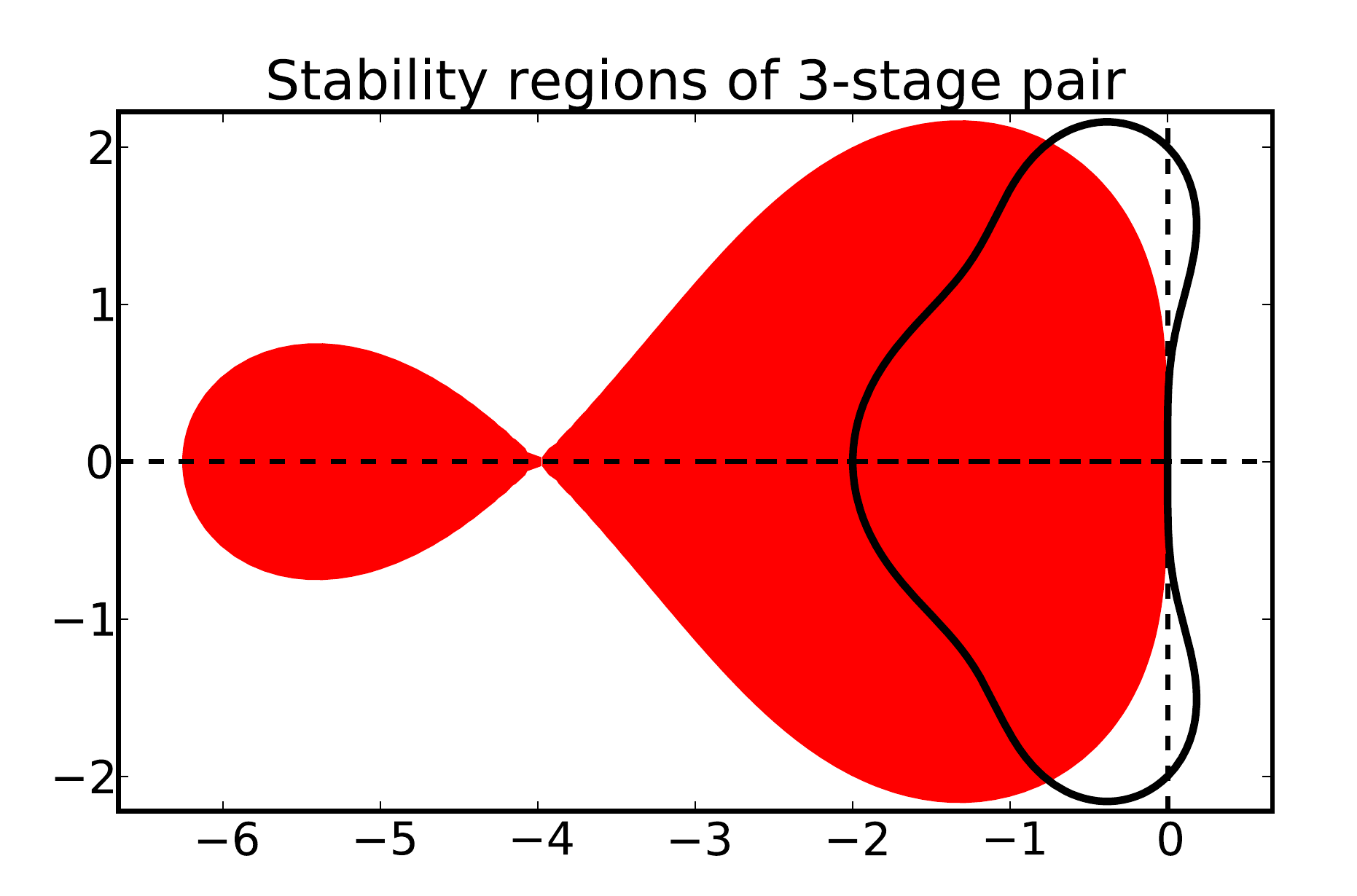}
    \includegraphics[width=0.45\textwidth]{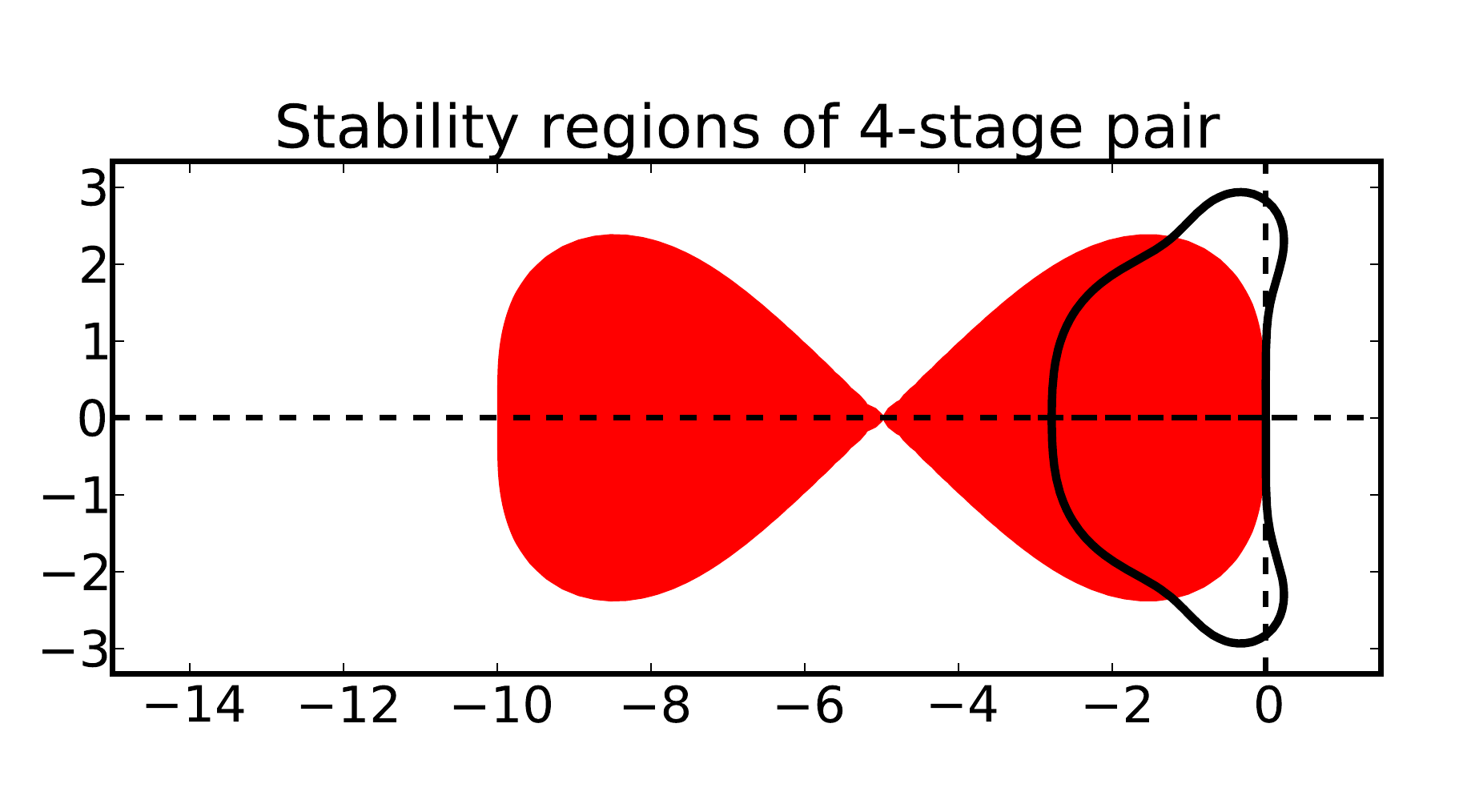}
  }
  \caption{Absolute stability regions for the three-stage embedded
    pair \eqref{eq:RKC32} (left) and the four-stage embedded pair
    \eqref{eq:RK4} (right).  The shaded red region corresponds to the
    method optimized for real axis inclusion and the black line
    corresponds to the method optimized for imaginary axis inclusion.
  }
  \label{fig:linear-stability}
\end{figure}

\subsection{Numerical results}
We now give the results of numerical experiments carried out using
our three- and four-stage SPERK schemes.
The partitioning parameter $\chi$ is set equal to 1 wherever
$a(x)$ is more than 0.005; elsewhere it is zero.
More generally, $\chi$ could be selected based on the local Reynolds number.
In all cases, we vary the time step-size and compute max-norm absolute
errors by comparing to a highly accurate RK4 approximation
of the spatially discretized system at the final time $t=0.1$.
All computations use a constant mesh spacing of $\Dx = 1/250.$  

The results for three-stage methods are given in the left-hand plot of Fig.~\ref{fig:advection_diffusion_results}.
We find that the scheme that is designed for advection-dominated flows
gives good results for $\Dt\le \scinot{1.45}{-5}$.   Stability is lost for larger time step-sizes.
The second method, RKC(3,2),  is  unstable for  $\Dt\ge \scinot{1.93}{-5}$.  
The SPERK scheme derived from the combination is stable for values of $\Dt$ that
are more than two times greater.

The right-hand plot of  Fig.~\ref{fig:advection_diffusion_results} gives the results for
four-stage methods.  
The classical fourth-order Runge--Kutta method RK4
has a small  error but becomes unstable for time steps larger than $\scinot{2}{-5}$.
Our variant of RKC(4,2) has a similar time step restriction but produces a larger error (because it is second order).
The SPERK scheme gives improved stability and allows for time steps that are more than three times larger.
In this example, whenever RK4 is stable, the SPERK scheme gives the
same accuracy because the largest error occurs in the
convection-dominated region.

For both embedded pairs, if we repeat the experiments with flux-based
partitioning, we observe essentially the same errors.  This is because
the largest errors occur away from the switching interface.

\begin{figure}
  \centerline{%
    \includegraphics[width=0.45\textwidth]{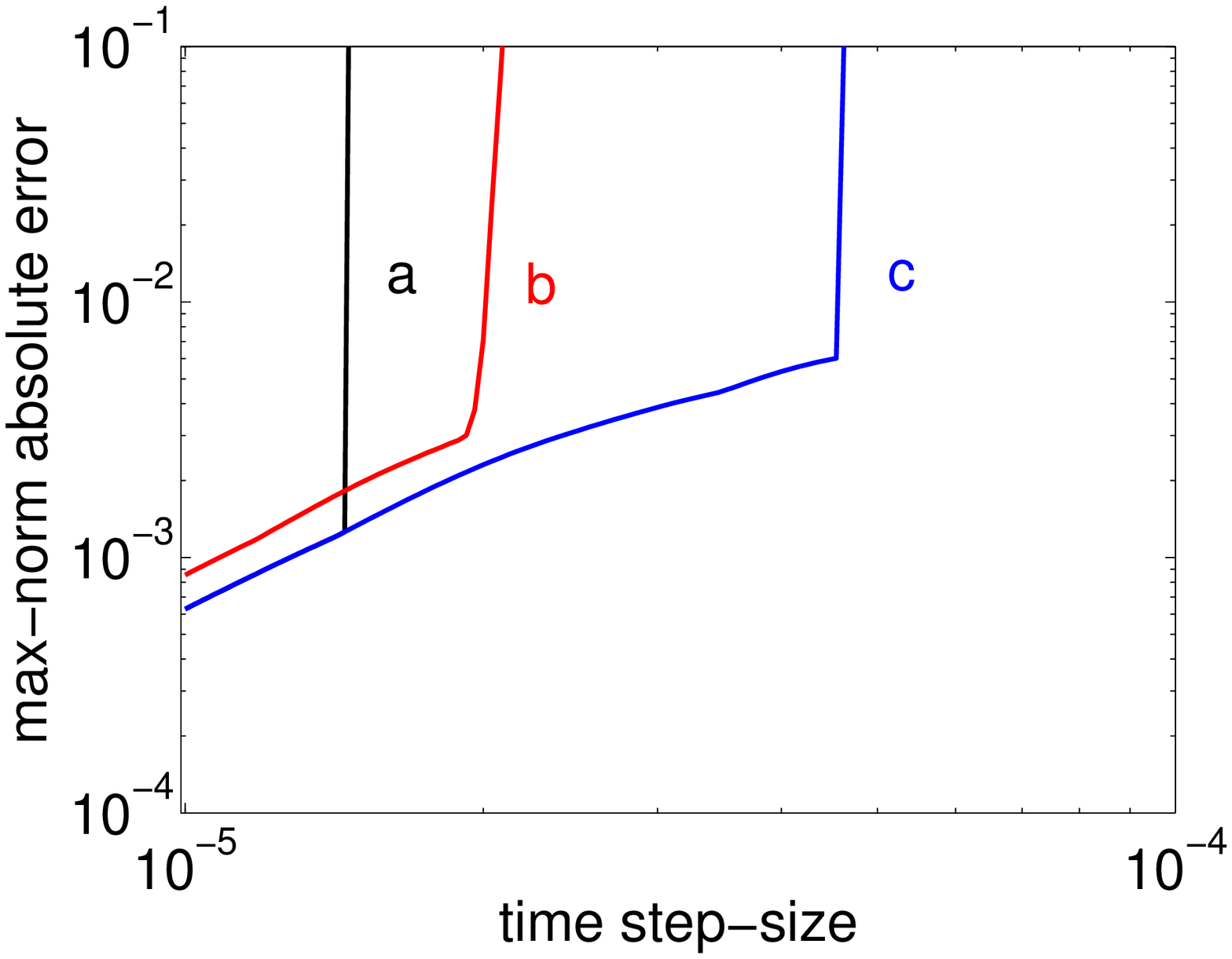}
    \qquad
    \includegraphics[width=0.45\textwidth]{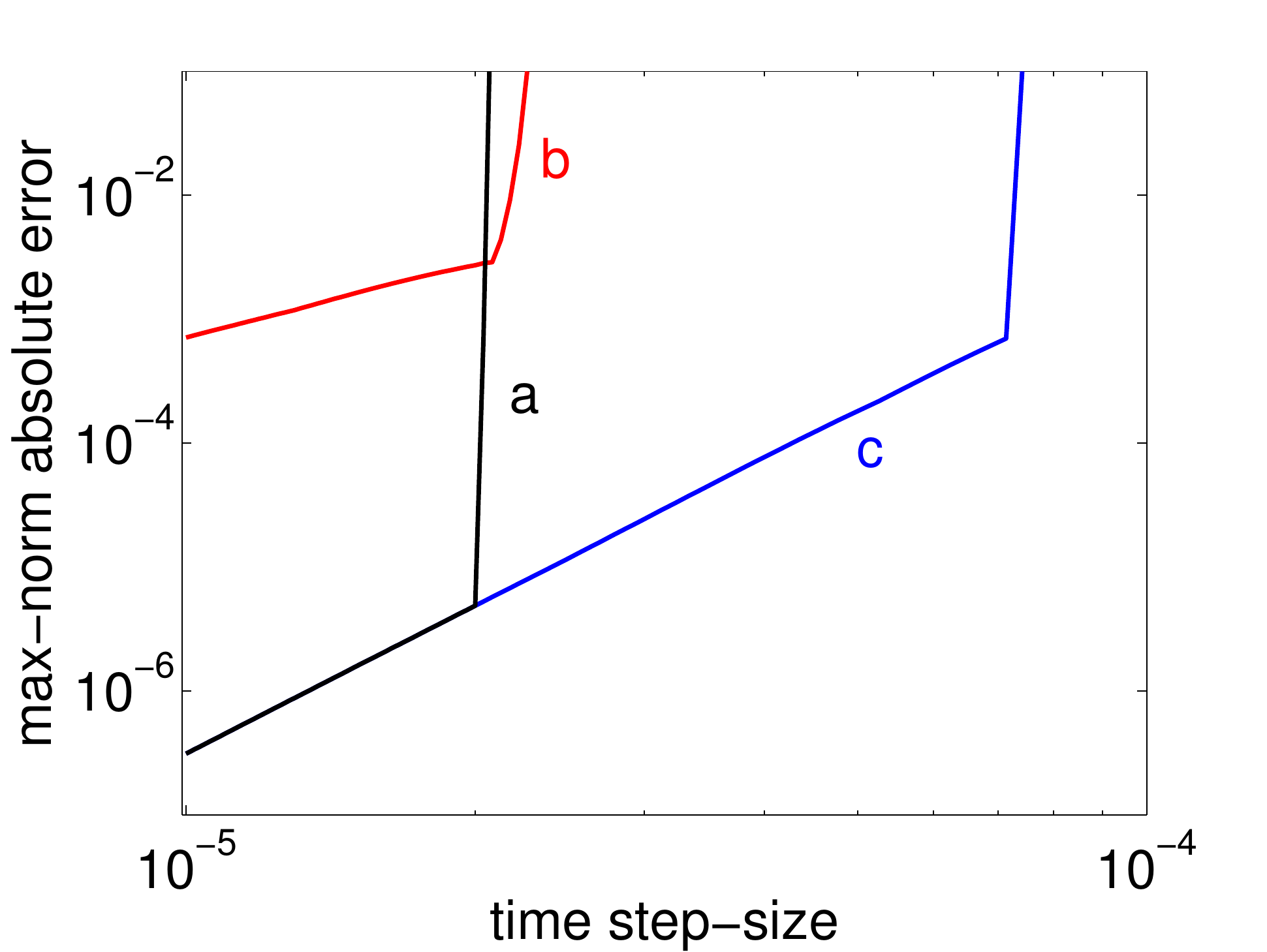}
  }
  \caption{Maximum norm absolute errors for SPERK schemes applied to the variable coefficient advection-diffusion problem.
The left plot gives results for three-stage schemes: (a) scheme with maximal imaginary axis inclusion, (b) RKC(3,2),
(c) equation-based partitioning.
The right plot gives results for four-stage schemes: (a) RK4, (b) our variant of RKC(4,2),
(c) equation-based partitioning.
  }
  \label{fig:advection_diffusion_results}
\end{figure}

\section{Example: Spatially partitioned time-stepping for WENO} \label{sec:WENO}

Weighted essentially non-oscillatory (WENO) spatial discretizations
involve an adaptive, data-dependent combination of several candidate
stencils to compute fluxes \cite{Shu:WENO_SiamReview}.
The most commonly used scheme provides fifth-order accuracy in smooth
regions of the solution and formally third-order spatial
discretizations near shocks or other discontinuities.
Assuming a positive flux function,
the finite difference WENO scheme computes the positive
flux differences in \eqref{eq:fluxdiff} with
\begin{align}
  f_{j+\hf} = \omega_0
  \left(\tfrac{2}{6} f_{j-2} -\tfrac{7}{6} f_{j-1} + \tfrac{11}{6} f_{j} \right)
  &+ \omega_1
  \left( - \tfrac{1}{6} f_{j-1} + \tfrac{5}{6} f_{j} + \tfrac{2}{6} f_{j+1} \right) \nonumber\\
  &+ \omega_2
  \left(\tfrac{2}{6} f_{j} + \tfrac{5}{6} f_{j+1} - \tfrac{1}{6} f_{j+2} \right), \label{eq:f+WENO}
\end{align}
where the weights $\omega_0$, $\omega_1$, and $\omega_2$ are
chosen by computing data-dependent smoothness indicators
\cite{Shu:WENO_SiamReview}.
In smooth regions, $\omega_{0,1,2}$ will be close to $\frac{1}{10}$,
$\frac{6}{10}$ and $\frac{3}{10}$ respectively whereas near nonsmooth
spatial features (such as a discontinuity), they adopt a binary choice
biasing the stencil away from the discontinuity
\cite{Shu:WENO_SiamReview}.

\subsection{An embedded pair for WENO}
We start with the SSPRK(5,3) scheme \cite{Spiteri/Ruuth:newclass}.
This will be the lower-order scheme which is used in spatial regions where the solution is not
smooth.
If we embed this method in a larger 7-stage, 5th-order method, we have the following
Butcher tableau
\begin{equation}
  \begin{scriptsize}
  \begin{array}{c|ccccccc}
    0           &         &         &         &         &         &         &  \\
    0.3773      &  0.3773 &         &         &         &         &         &  \\
    0.7545      &  0.3773 &  0.3773 &         &         &         &         &  \\
    0.7290      &  0.2430 &  0.2430 &  0.2430 &         &         &         &  \\
    0.6992      &  0.1536 &  0.1536 &  0.1536 &  0.2385 &         &         &  \\
    c_6         &  a_{61}  &  a_{62} &  a_{63}  &  a_{64} & a_{65}  &         &  \\
    c_7         &  a_{71}  &  a_{72} &  a_{73}  &  a_{74} & a_{75}  & a_{76}  &   \\
    \hline
    \rule{0pt}{1.05em}\hat{b}^T     &  0.2067 &  0.2067 &  0.1171 &  0.1818 &  0.2876 &       0 &    0 \\
    b^{T}       &  b_1    &  b_2    &   b_3   &   b_4   &   b_5   & b_6     & b_7
  \end{array} \label{eq:rk75ssp53_unknown_coeffs2}
  \end{scriptsize}
\end{equation}
where we display only a few digits of the SSPRK(5,3) coefficients.
The unknown coefficients will determine an RK(7,5) scheme that will be
optimized for linear stability and used in spatial regions where the
solution is smooth.

A technique for determining these coefficients by satisfying the order
conditions
is explored in \cite{cbm:msc} following a strategy designed in
\cite{Verner:1978}.
We attempt to maximize the linear stability properties of the RK(7,5)
scheme; these are determined by the stability polynomial which in this
case is parameterized by the coefficients of the $z^6$ and $z^7$
terms \cite{Hairer:ODEs2:ed2}, polynomial expressions in the
coefficients \eqref{eq:rk75ssp53_unknown_coeffs2} which we denote by
$\alpha_6$ and $\alpha_7$ respectively.
We find that specifying the SSPRK(5,3)
coefficients restricts the possible solutions to a straight line
through the $\alpha_6$--$\alpha_7$ space.
By examining the linear stability properties along this line (see
Fig.~\ref{fig:rk75_rho_rho2_rho3_ossp53_restrict}), we can choose one
of the degrees of freedom (the ``homogeneous polynomial'' $I_{65}$,
see \cite{cbm:msc} for details) to maximize the linear stability
properties of the resulting scheme.
\begin{figure}
  \centerline{%
    \includegraphics[height=24ex]{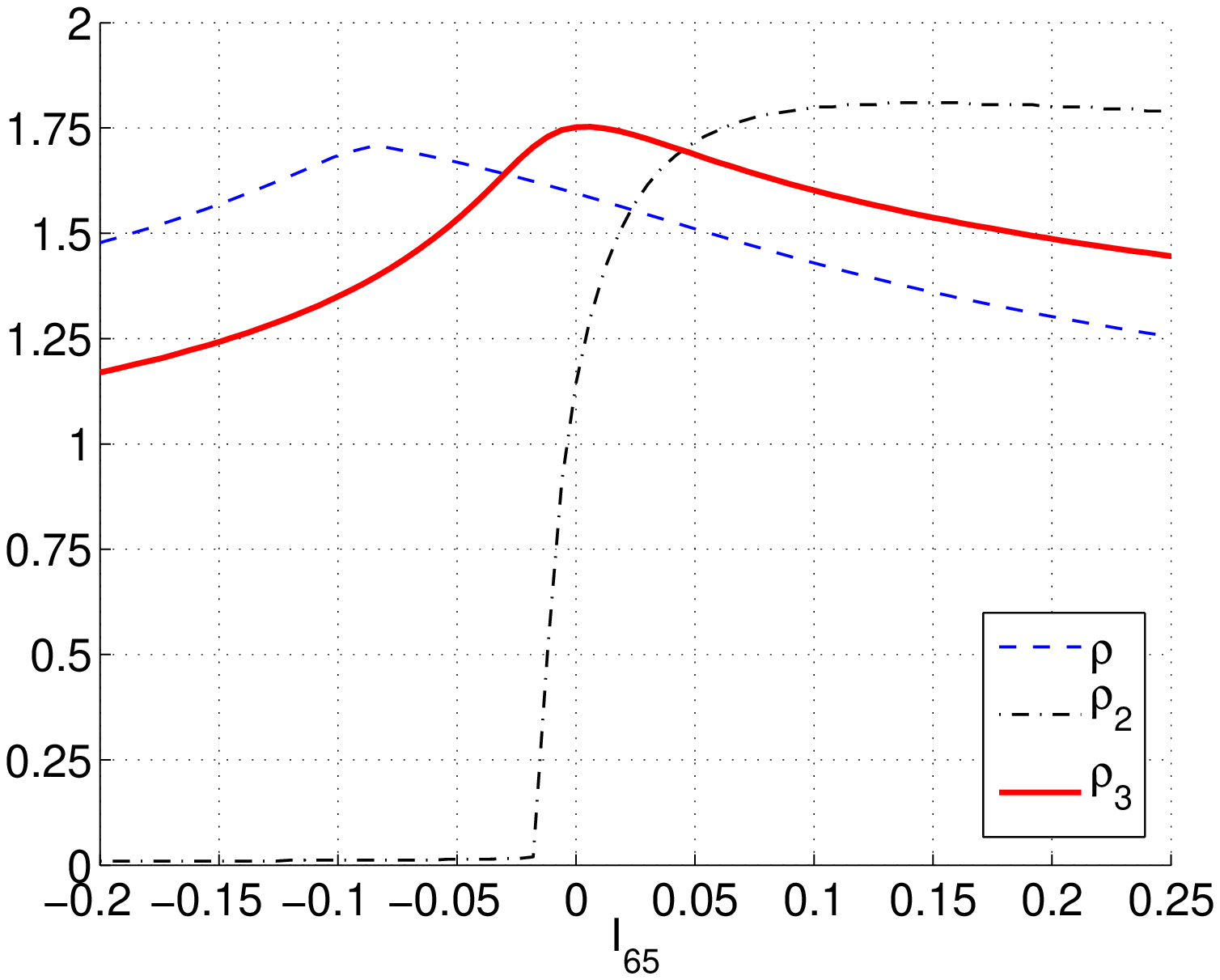}
    \qquad
    \includegraphics[height=24ex]{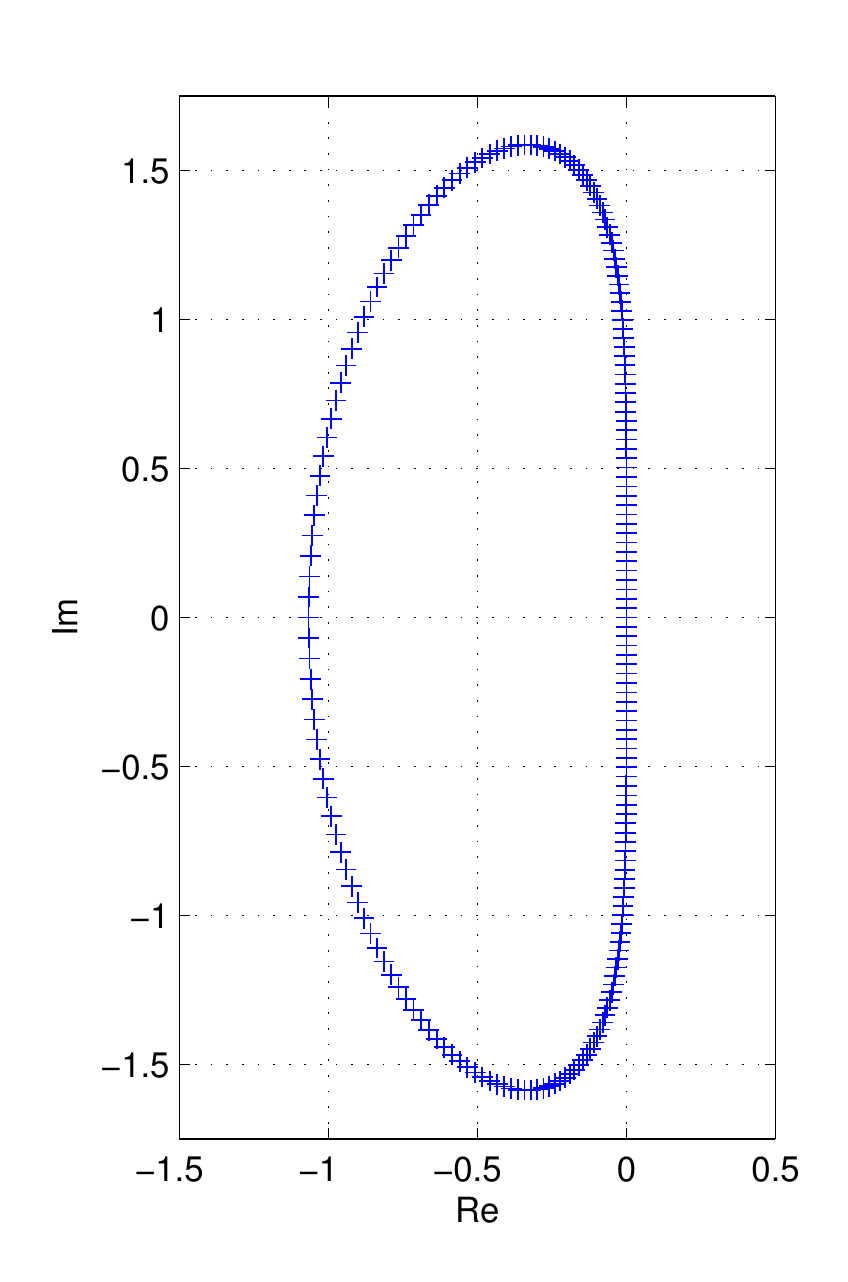}
  }
  \vspace*{-1.5ex}  
  \caption{Choosing the linear stability properties for RK(7,5)
    scheme.  Left: various measurements of the linear
    stability region versus a parameter $I_{65}$.
    Here $\rho$ is the radius of
    the largest inscribed disc and $\rho_2$ is the radius
    of the largest interval of the imaginary axis included.
    Right: the ``WENO bean'', the spectrum of the
    linearized WENO operator in smooth regions \cite{Jiang/Shu:1996,
      WangSpiteri:weno5LinearInstab, mohammad:weno5stab}:  $\rho_3$
    measures the largest scaling of this bean that will fit in the
    linear stability region.
    We choose $I_{65}=0.0037$ and find $\rho_3 \approx 1.76$,
    $\rho \approx 1.58$, and $\rho_2 \approx 1.2$.
  }
  \label{fig:rk75_rho_rho2_rho3_ossp53_restrict}
\end{figure}
There are still six (nonlinear) order conditions to satisfy and six
coefficients to be determined.
Solving these remaining equations with a computer algebra system
results in a single discrete solution and a one-parameter family of
solutions.
The one-parameter family was discarded because each member had either
a negative node $c_6$ or unreasonably large coefficients.
The lone solution did not have these deficiencies and we use it for our numerical
experiments.
The coefficients are shown to 15 digits in Table~\ref{tab:rk75ssp53_scheme}.

\begin{table}
  \caption{Coefficients of the embedded RK(7,5)/SSPRK(5,3) method.}
  \label{tab:rk75ssp53_scheme}
  \vspace*{-1.5ex}  
  \centering
  \begin{scriptsize}
    \begin{tabular}{lll}
      \hline
      $a_{21} = 0.377268915331368$, \\
      $a_{31} = 0.377268915331368$, &
      $a_{32} = 0.377268915331368$, \\
      $a_{41} = 0.242995220537396$, &
      $a_{42} = 0.242995220537396$, &
      $a_{43} = 0.242995220537396$, \\
      $a_{51} = 0.153589067695126$, &
      $a_{52} = 0.153589067695126$, &
      $a_{53} = 0.153589067695126$, \\
      $a_{54} = 0.23845893284629$, &
      $\hat{b}_1 = 0.206734020864804$, &
      $\hat{b}_2 = 0.206734020864804$, \\
      $\hat{b}_3 = 0.117097251841844$, &
      $\hat{b}_4 = 0.18180256012014$, &
      $\hat{b}_5 = 0.287632146308408$, \\
      \hline
      $a_{61} = 0.113015751552667$, &
      $a_{62} = 1.49947221487533$, &
      $a_{63} = 0.134753400626063$, \\
      $a_{64} = -1.06421259296782$, &
      $a_{65} = 0.205145170072233$, &
      $a_{71} = -0.512110930783855$, \\
      $a_{72} = 3.91735780781337$, &
      $a_{73} = -0.0470520461913835$, &
      $a_{74} = -0.218621292015928$, \\
      $a_{75} = -1.64543995945252$, &
      $a_{76} = -0.494133579369683$, &
      $b_1 = 0.122097569374901$, \\
      $b_2 = 0.492898173466563$, &
      $b_3 = -0.232023614650883$, &
      $b_4 = -1.98394581022939$,\\
      $b_5 = 1.85394392181784$, &
      $b_6 = 0.965538124667539$, &
      $b_7 = -0.21850836444657$.\\
    \hline
    \end{tabular}
  \end{scriptsize}
\end{table}

\subsection{Mask selection}  \label{sec:switching}
Various approaches to spatially switching between different schemes
have been proposed in the context of nonlinear hyperbolic PDEs;
see e.g.,~\cite{harabetian1993nonconservative} and references therein.
One particularly simple but effective approach used in~\cite{harabetian1993nonconservative}
is to consider the
second order differences of the solution:
\begin{align} \label{2diff-switch}
\chi_i & = \begin{cases}
1 & \text{if } |\Delta^2 q_i | < C \Dx^2 \\
0 & \text{if } |\Delta^2 q_i | \ge C \Dx^2.
\end{cases}
\end{align}
This approach sets $\chi=0$ (so the lower-order SSP scheme will be used)
wherever $\Delta^2 q$ is large.  In practice, setting $C=500$
seems to work well for the problems investigated below.

Alternatively, one may conveniently use the WENO weights $\omega_{0,1,2}$ to
select the mask $\chi$ in a WENO-based SPERK scheme.
One possible choice is
\begin{align} \label{weno-switch}
\chi_i &= \begin{cases}
1 & \text{if $|\omega_0 - \tfrac{1}{10}| \leq 0.06$
         and $|\omega_1 - \tfrac{6}{10}| \leq 0.06$
         and $|\omega_2 - \tfrac{3}{10}| \leq 0.06$,}  \\
0 & \text{otherwise,}
\end{cases}
\end{align}
where $0.06$ is an adjustable threshold parameter.
That is, if the WENO weights are within this small threshold of their
theoretical smooth-region values then we propagate the higher-order
linearly stable RK(7,5) solution, otherwise we propagate the
lower-order SSPRK(5,3) solution.
We compute the mask based on the weights corresponding to $U^n$, the solution at time $t_n$ and thus fix the mask over each time step.

With either approach, we then set
\begin{align} \label{expand_mask}
\chi_i = \min_{|j|\le 4} \chi_{i+j};
\end{align}
i.e., the mask is widened by four grid cells to increase the usage of the SSP scheme.
This was suggested in \cite{harabetian1993nonconservative} and is essential
to ensure that a shock does not leave the zone in which the SSP method is used during one time step.

\paragraph{Remark} There are many possibilities for choosing the mask
$\chi$; here we have suggested two simple approaches but this not
intended to be exhaustive, nor have we tried to optimize the 0.06
threshold parameter or the value of $C$.

\subsection{Burgers' equation}  \label{sec:burgers}
We perform convergence studies which demonstrate that
the SPERK methods achieve the predicted orders of accuracy.
Our tests use the inviscid Burgers' equation $u_t + (u^2)_x = 0$ on
the periodic domain $[-1, 1]$ with smooth initial conditions
\begin{equation} \label{eq:smoothICs}
  u(x,0) = \hf - \hf \cos \left(\pi \left(x - \frac{\sin 2\pi x}{4\pi}\right) \right).
\end{equation}
Table~\ref{tab:conv} shows that the methods achieve (at least) their
expected orders of accuracy.  Namely, five when the fifth-order scheme
is used ($\chi(x) = 1$ for all $x$) and three when using the
third-order scheme at any subset of points.
As noted in Section~\ref{sec:blending}, $\chi$ need not be a binary
choice: we see that the method maintains third-order even when random
values of $\chi$ (in the range $[0,1]$) are selected at each point
and at each time-step.
Note also that  the flux-partitioned scheme exhibits an order of
accuracy $\min(p,\hat{p}) = 3$, \emph{without} the loss of one order
predicted by our Theorem~\ref{thm:mixtimeOrder:flux-based}; this is
typical of what we observed in all our numerical tests.

\begin{table}
  \caption{Convergence studies for various mask functions
    $\chi(x)$, performed on
    Burgers' equation with smooth initial conditions
    discretized using WENO and the SPERK 5th-order/3rd-order pair.
    We compute to $\tfinal = 0.25$ using a CFL number of 1.2; error measured
    in an approximate $L_2$ norm.
  }
  \label{tab:conv}
  \centering
  \begin{small}
  \begin{tabular}{cccccc}
    \hline
    & & \multicolumn{4}{c}{error for different $\chi(x)$ functions} \\
    partitioning & $\Delta x$ & $1$ &
    $0$  &
    Heaviside$(x)$  &
    random(0,1) \\
    \hline
    equation-  & 1/320 & \rule{0pt}{1.05em}\scinot{1.29}{-7}  & \scinot{5.59}{-7} & \scinot{1.29}{-7} & \scinot{3.33}{-7} \\
    based    & 1/640 & \scinot{4.09}{-9}  & \scinot{7.00}{-8} & \scinot{4.44}{-9} & \scinot{3.76}{-8} \\
    &                1/1280 & \scinot{1.29}{-10} & \scinot{8.64}{-9} & \scinot{5.48}{-10} & \scinot{4.39}{-9} \\
    est.~order      &        & 4.99 & 3.02 &  3.45 & 3.08 \\
    \hline
    flux-based   & 1/320 & \rule{0pt}{1.05em}\scinot{2.01}{-8}  & \scinot{1.08}{-7} & \scinot{2.44}{-8}  & \scinot{6.63}{-8} \\
                 & 1/640 & \scinot{6.38}{-10} & \scinot{1.34}{-8} & \scinot{1.87}{-9}  & \scinot{8.06}{-9} \\
    &             1/1280 & \scinot{2.00}{-11} & \scinot{1.65}{-9} & \scinot{2.18}{-10} & \scinot{1.05}{-9} \\
    est.~order   &       & 4.99 & 3.02 & 3.44 & 2.99 \\
    \hline
  \end{tabular}
  \end{small}
\end{table}
In addition to these finite difference convergence studies,
we also tested the equation-based SPERK scheme with a
finite-volume WENO code \cite{pyclaw} and
observed similar results.

\subsubsection{Total variation tests}  \label{sec:tvd}
We consider the discrete total variation seminorm of the solution on
two test problems shown in Fig.~\ref{fig:burg_profile}.
The first consists of the inviscid Burgers' equation with a
square wave initial condition in $[0,1]$.
We determined the largest CFL number $\sigma$, $\dt = \sigma \dx$, such that
the resulting solution experiences no significant increase in the TV seminorm.
For the flux-partitioned scheme
we find in Fig.~\ref{fig:TVinc}
that the SSPRK(5,3) scheme
(i.e., $\chi = 0$ everywhere) is no longer TVD for $\sigma > 1.4$.
The TV error of the RK(7,5) solution slowly increases with $\sigma$
and, for example, exhibits a TV error larger than $10^{-4}$ for
$\sigma > 0.5$.
With the SPERK embedded pair (using $\chi$ as chosen by the WENO
weights
using \eqref{weno-switch} and \eqref{expand_mask}),
the loss of the TVD
property occurs for $\sigma > 1.4$ (that is, the same as the
SSPRK(5,3) scheme).
The results using equation-based partitioning are very similar.

\begin{figure}
  \centerline{%
    \includegraphics[width=0.45\textwidth]{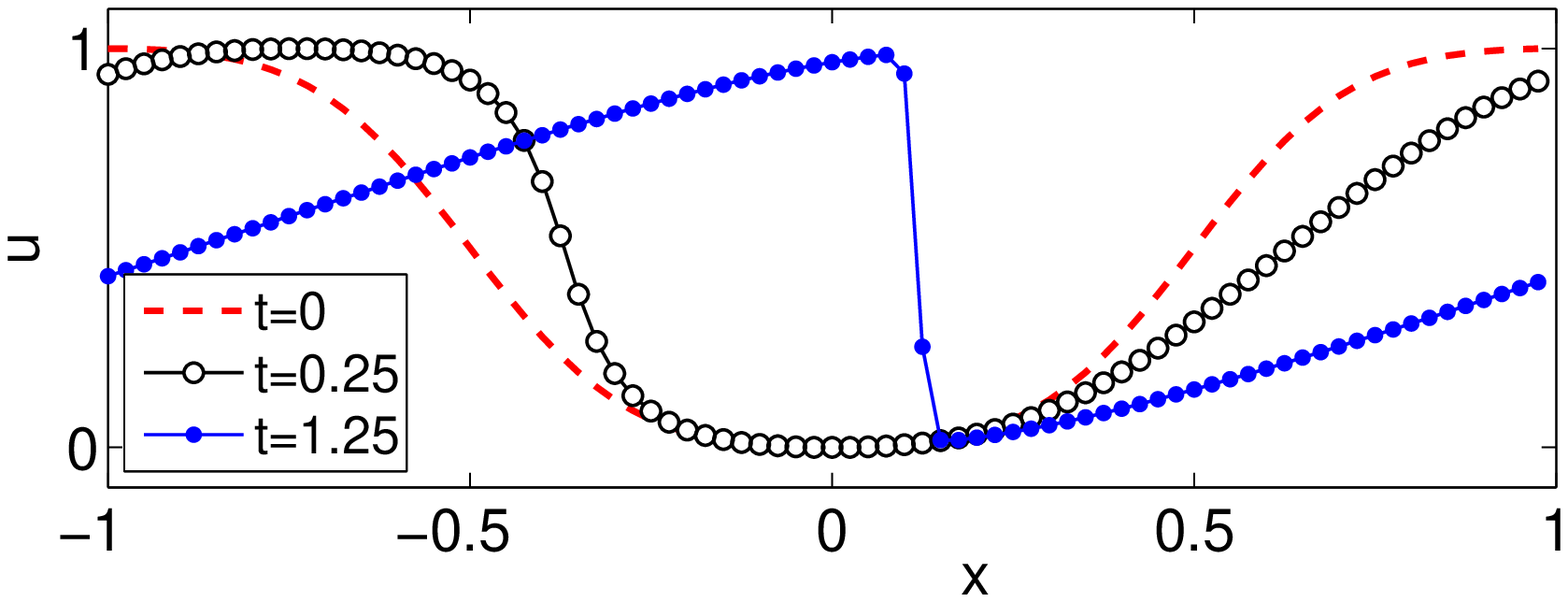}%
    \qquad
    \includegraphics[width=0.45\textwidth]{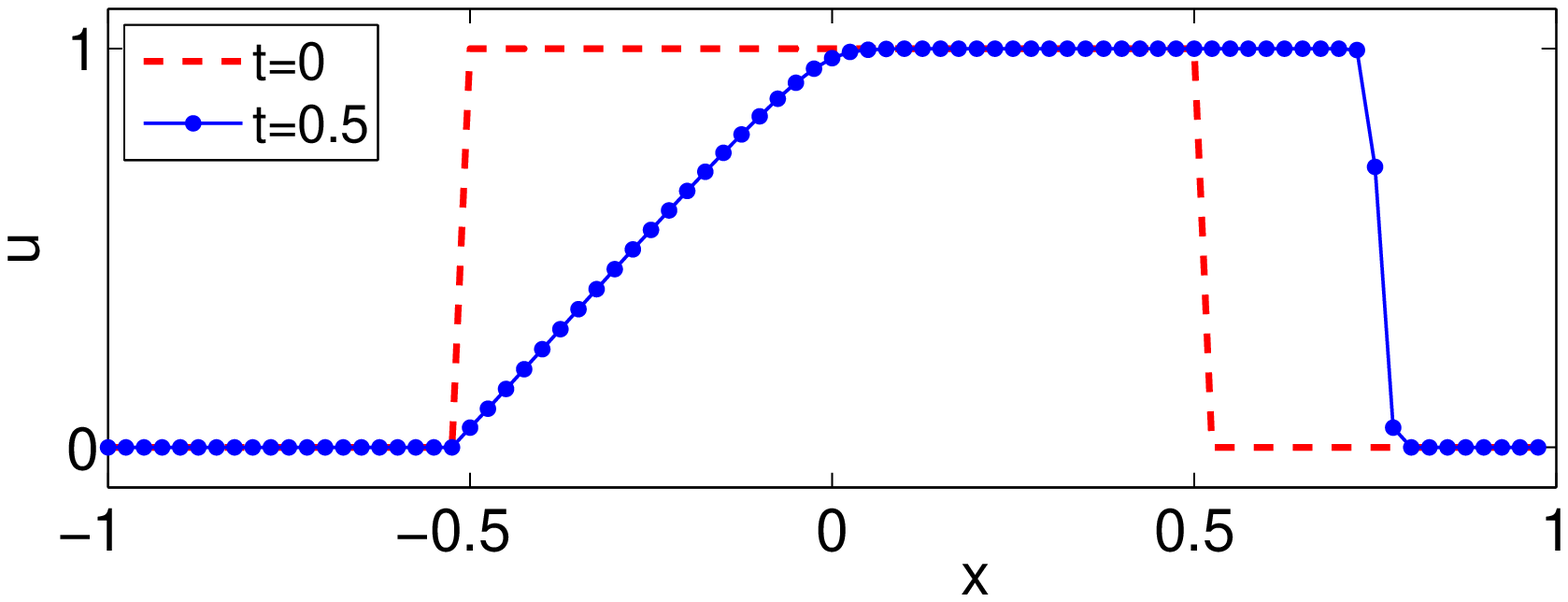}%
  }%
  \vspace*{-1.5ex} 
  \caption{Solution profiles for Burgers' equation with the two
    initial conditions used in \S\ref{sec:burgers}: an initially
    smooth curve that evolves into a shock (left) and shock and
    expansion fan from a square wave initial condition (right).
    Solutions computed with the WENO-based SPERK method using $\dx =
    0.025$.}
  \label{fig:burg_profile}
\end{figure}

\begin{figure}
  \centerline{%
    \includegraphics[width=0.45\textwidth]{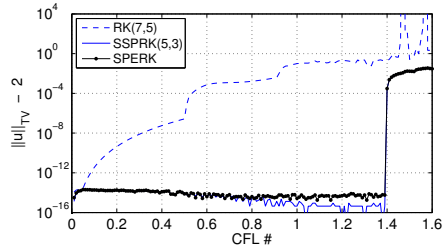}%
  }%
  \vspace*{-1.5ex} 
  \caption{Total variation increase
      for Burgers' equation with square wave initial
      conditions (initial total variation $2$), $\dx=0.025$,
      $\tfinal = 0.5$, and WENO regularization parameter
      $\epsilon = 10^{-30}$.
      Note that for most CFL numbers, the SPERK curve overlays the SSPRK(5,3) results.}
  \label{fig:TVinc}
\end{figure}

Next consider the inviscid Burgers' equation with smooth initial
conditions~\eqref{eq:smoothICs}.
The WENO weights are again used to choose $\chi$.
Table~\ref{tab:weno_bestof} shows that the embedded pair
SPERK scheme offers high-order accuracy when the solution is
smooth and remains TVD later when the solution is non-smooth; the best of
both worlds!

\begin{table}
  \caption{The flux-partitioned SPERK scheme combined
    with WENO has the best features of each method.
    Here the initial smooth curve
    sharpens into a shock.
    At small times $t=0.25$ before shock formation, both of
    the underlying schemes exhibit their design orders of 5 and 3
    respectively.
    At a later time $t=1.25$, a shock has formed; all schemes
    exhibit larger errors as the solution is no longer smooth.
    However, RK(7,5) now exhibits spurious oscillations,
    indicated by an increase in the total variation seminorm.
    We use a CFL number of 1.2,
    $\Dx = \frac{1}{320}$, and error measured in an approx.~$L_2$ norm.
  }
  \label{tab:weno_bestof}
  \centering
  \begin{small}
  \begin{tabular}{cccccc}
    \hline
    method & time & error & error ($\Dx/2$) & est.~order & TV increase \\
    \hline
    RK(7,5)    & $t=0.25$ & 2.01e-8 &  6.38e-10 & 4.98 & 0 \\
            & $t=1.25$ & 8.72e-3  & 8.79e-3   & --- &  0.264  \\
    \hline
    SSPRK(5,3)   & $t=0.25$ & 1.08e-7 & 1.34e-8 & 3.01 & 0 \\
            & $t=1.25$ & 1.35e-3  & 9.53e-4  &   0.50    & 0  \\
    \hline
    SPERK & $t=0.25$ & 2.01e-8  & 6.38e-10    & 4.98 & 0 \\
            & $t=1.25$ & 1.35e-3 & 9.53e-4  &  0.50  & 0 \\
    \hline
  \end{tabular}
  \end{small}
\end{table}

\subsection{Euler equations: conservation and equation-based partitioning}
As we have seen, equation-based partitioning is non-conservative and can lead
to incorrect shock speeds.  It is suggested in \cite{harabetian1993nonconservative} that 
correct shock speeds should be obtained as long as the switching between schemes
occurs in regions where the solution is smooth.

We consider the Euler equations of compressible fluid dynamics; specifically, we solve
the Shu--Osher problem of \cite{shu1989efficient}, in which a shock wave impacts
a sinusoidally-varying density field.  We use the
fifth-order WENO wave-propagation method of \cite{Ketcheson2011} and a standard Roe Riemann solver
with entropy fix.
First, integrating with the 5th-order RK method, the solver fails for
any CFL number greater than $0.89$, due to the appearance of negative densities near
the shock in the sixth Runge--Kutta stage of the first time step.
Next, using the third-order SSP(5,3) method,
the integration is successful using a CFL number of $1.2$.
The accuracy obtained with either method is similar, with the largest errors occurring 
at the shock.

Finally, the problem is solved using the embedded pair, with switching based on
\eqref{2diff-switch} and applying \eqref{expand_mask} to ensure that switching does not occur
too near the shock.  A sample solution and a plot of $\chi$ are shown in Fig.~\ref{fig:SOshock}.
The SPERK method propagates the shock correctly and captures
the high-frequency region behind the shock as accurately as either of the component schemes.

\begin{figure}
  \centerline{%
    \includegraphics[width=0.48\textwidth]{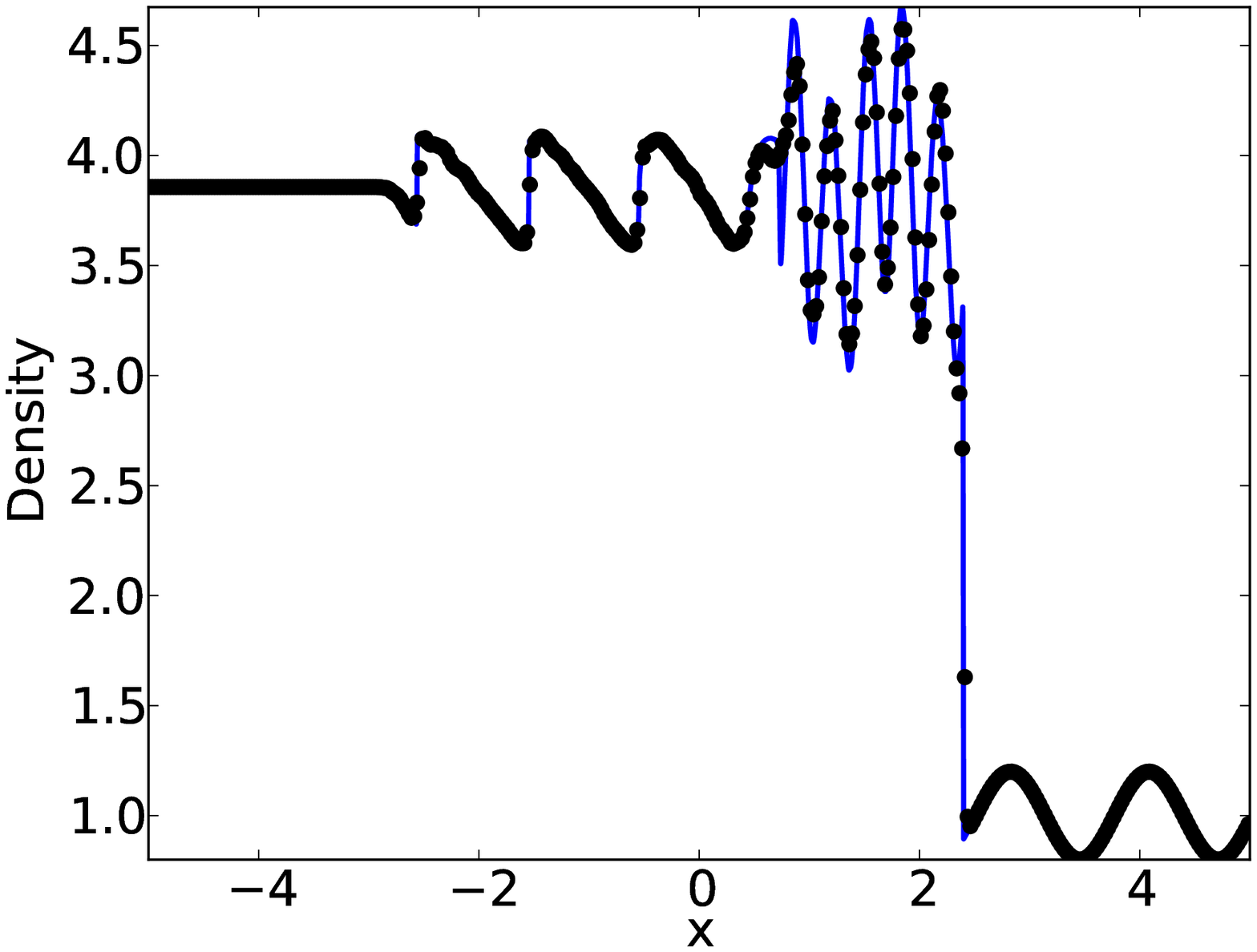}
    \includegraphics[width=0.48\textwidth]{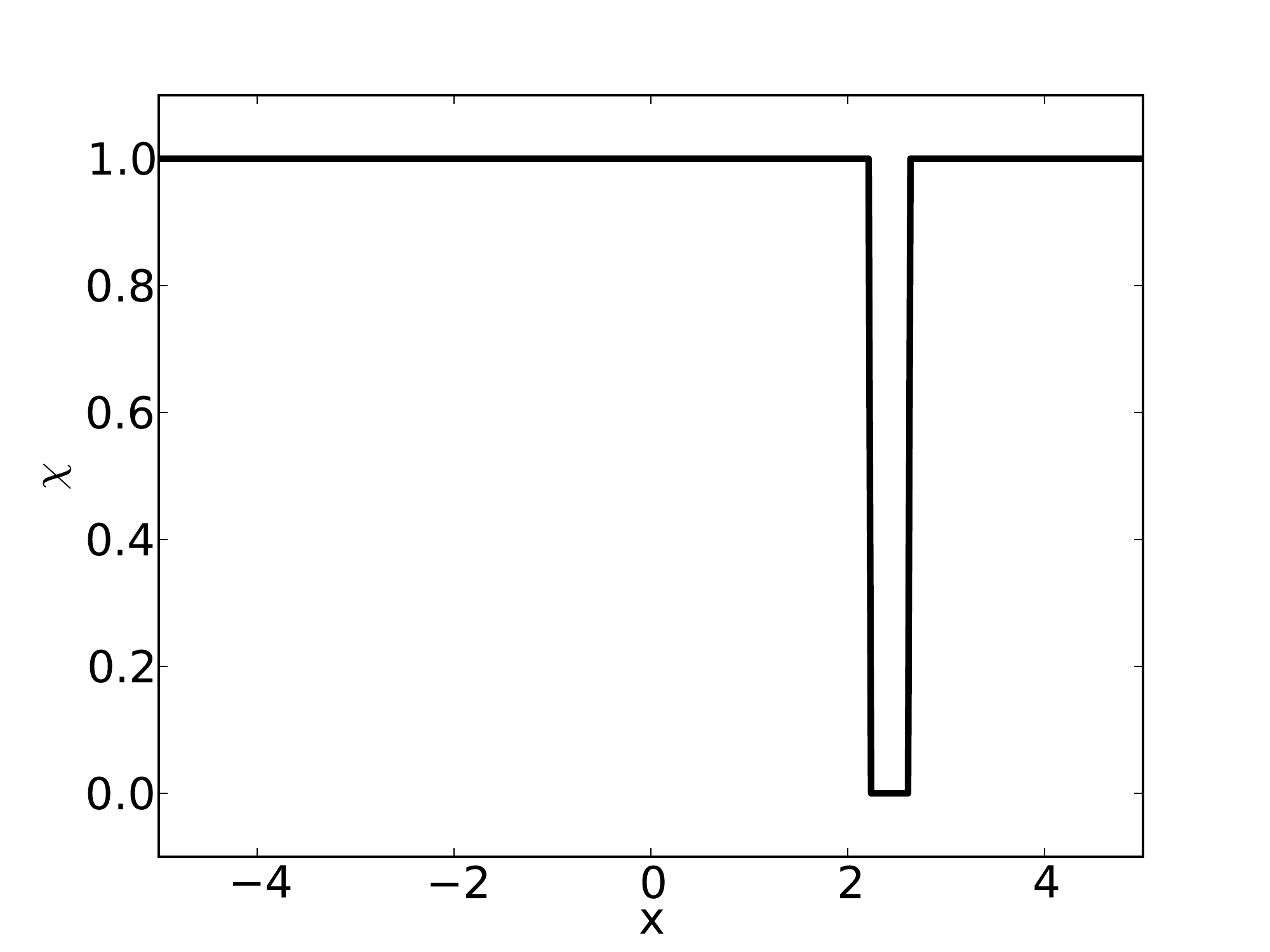}
  }
  \caption{Solution of the Shu--Osher problem at time $t=1.8$.  Left: the solid line represents a reference solution
  computed with just the SSP scheme and 6400 points; the black circles are computed using the 
  embedded pair and 400 points with CFL number 1.2.  Right: The mask $\chi$, showing that the SSP scheme is used only
  in a small region around the shock.
  }
  \label{fig:SOshock}
\end{figure}

\section{Discussion and future work\label{sec:discussion}}
This paper considers spatially partitioned embedded Runge--Kutta (SPERK) schemes.  Such methods give an efficient means to 
apply two or more Runge--Kutta methods to a spatially discretized PDE.  In this paper, SPERK schemes
are applied in two ways.  The first of these partitions the ODEs, and
is therefore referred to as equation-based partitioning.  In these methods 
the order of accuracy is shown to equal the minimum of the order of accuracy of the schemes composing the embedded pair.
Equation-based partitioning is not conservative when applied to a conservation law, however.  The second partitioning method
partitions by fluxes.  This flux-based partitioning has the advantage of being conservative when applied to a conservation law.
Theoretically, it may lead to a loss of one order of accuracy when compared to the underlying schemes; however, 
this loss of accuracy is not observed in practice. 
We show that both equation- and flux-based partitioning are positivity-preserving under a suitable time step restriction
when the underlying schemes are strong-stability-preserving (SSP).  
Numerical experiments on a spatially discretized 
variable coefficient advection-diffusion
equation show that  SPERK schemes can have superior linear stability properties than either scheme
individually.   
SPERK schemes may be applied to weighted non-oscillatory (WENO) spatial discretizations of conservation laws.
One approach to partitioning is carried out using the weights introduced in the WENO spatial discretization.
Specifically, a fifth-order Runge--Kutta method is used in smooth regions where WENO chooses a fifth-order stencil, 
and a third-order SSP Runge--Kutta method is used in non-smooth regions where WENO
chooses a formally third-order stencil.  We find that this combination avoids oscillations in
problems with shocks, and gives fifth-order accuracy in smooth flows.   

As part of this paper, several explicit SPERK schemes are designed and numerically validated.
Future work will carry out a more systematic development of  SPERK schemes.  Of particular interest for us
are methods that have good monotonicity where the solution is nonsmooth, and good accuracy elsewhere.
Such methods are particularly attractive for approximating conservation laws with isolated shocks and other nonsmooth
features.  

Schemes with improved local absolute stability properties (like those designed in
Section~\ref{sec:advectiondiffusion}) seem promising for problems with strong
spatial variation in the coefficients or grid size.  The stability polynomial optimization
algorithm developed in \cite{Ketcheson_Ahmadia_2012} could be used more generally to design
appropriate methods for such problems.

Theoretically, the accuracy of flux-based partitioning is expected to be less than the accuracy of 
the schemes composing the embedded pair but this loss of accuracy is not seen in our tests.  
This effect may be similar to what is observed in multirate
schemes.  For example, the Osher--Sanders scheme \cite{OsherSanders} is locally inconsistent, but has been found to
give first-order convergence in practice.
An explanation for this is given in \cite{hundsdorfer2012monotonicity} where it is shown that
local inconsistencies may not show up in the global errors due to cancellation and damping effects.
We hope to transfer this analysis to the case of SPERK schemes to obtain a better understanding
of the unexpectedly good performance of flux-based partitioning.

We have not attempted to prescribe general techniques for selecting the function $\chi$,
since those must depend on the nature of the problem and the purpose of switching between
schemes.  In the context of switching based on smoothness of the solution,
simple but useful approaches have been proposed in \cite{harten1972self,harabetian1993nonconservative}.
Investigation of effective switching methods for specific classes of problems is
left to future work.

\bibliographystyle{siam}
\bibliography{cbm,mixtime}

\end{document}